\documentclass[aap,preprint]{imsartPM}

\usepackage{amsmath}
\usepackage{amssymb}
\usepackage{amsfonts}
\usepackage{amsthm}

\usepackage[pdftex]{graphicx}
\usepackage{subfigure}
\usepackage{tikz}
\usetikzlibrary{calc}
\usetikzlibrary{arrows}

\usepackage[round,sort&compress,longnamesfirst]{natbib}

\usepackage{paralist}
\usepackage{wasysym}

\usepackage[T1]{fontenc}

\usepackage[mathcal]{euscript}

\newcommand{\dd}{\,d}
\newcommand{\eps}{\varepsilon}

\newcommand{\ip}{\!\cdot\!}
\newcommand{\bigoh}{\mathcal O}
\newcommand{\pd}{\partial}
\newcommand{\injects}{\hookrightarrow}

\newcommand{\simplex}{\Delta}
\newcommand{\set}{\mathcal{S}}
\newcommand{\game}{\mathfrak{G}}
\newcommand{\play}{\mathcal{N}}
\newcommand{\act}{\mathcal{A}}
\newcommand{\strat}{\Delta}
\newcommand{\graph}{\mathcal{G}}
\newcommand{\edges}{\mathcal{E}}
\newcommand{\nodes}{\mathcal{V}}
\newcommand{\net}{\mathcal{Q}}
\newcommand{\model}{\mathfrak{C}}
\newcommand{\eq}{\strat\!^{*}}
\newcommand{\pair}{\sigma}

\newcommand{\dkl}{d_{\text{\normalfont KL}}}
\newcommand{\gen}{\mathcal{L}}

\newcommand{\R}{{\mathbb R}}

\DeclareMathOperator{\exclude}{\setminus}
\DeclareMathOperator{\prob}{\mathbf{P}\!}
\DeclareMathOperator{\ex}{\mathbf{E}}

\DeclareMathOperator{\Int}{Int}
\DeclareMathOperator{\bd}{bd}
\DeclareMathOperator{\ess}{ess}
\DeclareMathOperator{\red}{red}
\DeclareMathOperator{\im}{im}
\DeclareMathOperator{\supp}{supp}
\DeclareMathOperator{\class}{class}

\newcommand{\txs}{\textstyle}
\newcommand{\insum}{\sum\nolimits}

\newcommand{\negspace}{\!\!\!}

\definecolor{blueviolet}{rgb}{0.25,0.1,.5}
\definecolor{olivegreen}{rgb}{0,0.5,.2}
\definecolor{orange}{rgb}{1,0.4,0}

\theoremstyle{plain}
\newtheorem{theorem}{Theorem}
\newtheorem{corollary}[theorem]{Corollary}
\newtheorem*{corollary*}{Corollary}
\newtheorem{lemma}[theorem]{Lemma}
\newtheorem{proposition}[theorem]{Proposition}

\theoremstyle{definition}
\newtheorem{definition}[theorem]{Definition}
\newtheorem*{definition*}{Definition}

\theoremstyle{remark}
\newtheorem{remark}{Remark}
\newtheorem*{remark*}{Remark}

\numberwithin{equation}{section}
\numberwithin{theorem}{section}
\begin{document}

\begin{frontmatter}

\begin{aug}
\title{Balancing Traffic in Networks:
Redundancy, Learning and the Effect of Stochastic Fluctuations}
\runtitle{Balancing Traffic in Networks}

\author{
\fnms{Panayotis}
\snm{Mertikopoulos}
\thanksref{t1}
\ead[label=e1]{pmertik@phys.uoa.gr}}
\and
\author{
\fnms{Aris}
\snm{L.}
\snm{Moustakas}
\thanksref{t1}
\ead[label=e2]{arislm@phys.uoa.gr}}

\runauthor{P. Mertikopoulos and A. L. Moustakas}

\affiliation{National \& Kapodistrian University of Athens}

\thankstext{t1}{Supported in part by the European Commission under grants EU-FET-FP6-IST-034413 (Net-ReFound); the first author was also supported by the Empirikeion Foundation of Athens, Greece.}

\end{aug}

\begin{abstract}
We study the distribution of traffic in networks whose users try to minimise their delays by adhering to a simple learning scheme inspired by the replicator dynamics of evolutionary game theory. The stable steady states of these dynamics coincide with the network's {\em Wardrop equilibria} and form a convex polytope whose dimension is determined by the network's {\em redundancy} (an important concept which measures the ``linear dependence'' of the users' paths). Despite this abundance of stationary points, the long-term behaviour of the replicator dynamics turns out to be remarkably simple: every solution orbit converges to a Wardrop equilibrium.

On the other hand, a major challenge occurs when the users' delays fluctuate unpredictably due to random external factors. In that case, interior equilibria are no longer stationary, but {\em strict} equilibria remain stochastically stable irrespective of the fluctuations' magnitude. In fact, if the network has no redundancy and the users are patient enough, we show that the long-term average of the users' traffic flows converges to the vicinity of an equilibrium, and we also estimate the corresponding invariant measure.
\end{abstract}



\end{frontmatter}


\section{Introduction}
\label{sec:introduction}

The underlying problem of managing the flow of traffic in a large-scale network is as simple to state as it is challenging to resolve: given the rates of traffic generated by the users of the network, one is asked to identify and realise the most ``satisfactory'' distribution of traffic among the network's routes.

Of course, given that this notion of ``satisfaction'' depends on the users' optimisation criteria, it would serve well to keep a concrete example in mind. Perhaps the most illustrative one is that of the Internet itself, where the primary concern of its users is to minimise the travel times of their data flows. However, since the time needed to traverse a link in the network increases (nonlinearly even) as the link becomes more congested, the users' concurrent minimisation efforts invariably lead to game-like interactions whose complexity precludes even the most rudimentary attempts at coordination. In this way, a traffic distribution will be considered ``satisfactory'' by a user when there is no unilateral move that he could make in order to further decrease the delays (or {\em latencies}) that he experiences.

This Nash-type condition is aptly captured by {\em Wardrop's principle} \citep{Wa52}: given the level of congestion caused by other users, every user seeks to employ the minimum-latency path available to him. As might be expected, this principle has attracted a great deal of interest and it was shown early on that these {\em Wardrop equilibria} can be calculated by solving a convex optimisation problem \citep{BMW56, DS69}. Among others, this characterisation enabled \cite{RT02,RT04} to quantify the efficiency of these equilibrial states by estimating their ``price of anarchy'', i.e. the ratio between the aggregate delay of a flow at Wardrop equilibrium and the minimum achievable (aggregate) latency \citep{KP99}.

Still, the size of large-scale networks makes computing these equilibria a task of considerable difficulty, clearly beyond the users' individual deductive capabilities. Moreover, a user has no incentive to actually play out his component of an equilibrial traffic allocation unless he is convinced that his opponents will also employ theirs (an argument which gains additional momentum if there are multiple equilibria). It is thus more reasonable to take a less centralised approach and instead ask: {\em is there a simple learning procedure which leads users to Wardrop equilibrium?}

\smallskip

Even though the static properties of Wardrop equilibria have been studied quite extensively, this question has been left relatively unexplored. In fact, it was only recently that the work of \cite{Sa01} showed that a good candidate for such a learning scheme would be the {\em replicator dynamics} of evolutionary game theory, a dynamical system that was first introduced by \cite{TJ78} to model the evolution of (nonatomic) populations that interact with one another by means of random matchings in a Nash game. More precisely, these dynamics arise as the byproduct of an ``imitation of the fittest'' process which drives the per capita growth rate of a genotype (strategy) proportionately to the difference between the reproductive fitness (payoff) of the genotype itself and the population average. Thus, owing to this correlation between growth rates and payoffs, the game's Nash equilibria emerge as $\omega$-limit points of the replicator trajectories \textendash\ see also the excellent surveys by \cite{We95} and by \cite{HS98,HS03}.


In our congestion setting, these populations correspond to the users' traffic flows, so the convex optimisation formulation of \citeauthor*{BMW56} allows us to recast our problem in terms of a (nonatomic) {\em potential game} \citep{Sa01}. Indeed, Wardrop equilibria can be located by looking at the minimum of the {\em Rosenthal potential} \citep{Ro73} and, hence, Sandholm's analysis shows that they are Lyapunov stable rest points of the replicator dynamics. This fact was also recognized independently by \cite{FV04} who additionally showed that the (interior) solution orbits of the replicator dynamics converge to the set of Wardrop equilibria \textendash~actually, the authors suggest that these orbits converge to a point, but their analysis only holds when there is a unique equilibrium.

Rather surprisingly, when there is not a unique equilibrium, the structure of the Wardrop set itself seems to have been overlooked in the above considerations. Specifically, it has been widely assumed that if the network's delay functions are strictly increasing, then there exists a unique Wardrop equilibrium \citep[for instance, see][Corollary 5.6]{Sa01}. As a matter of fact, this uniqueness property is only true in {\em irreducible} networks, i.e. networks whose paths are ``independent'' of one another (in a sense made precise by Definition \ref{def:redundancy}). In general, the Wardrop set of a network is a convex polytope whose dimension is determined by the network's {\em redundancy}, a notion which quantifies precisely this ``linear dependence''. Nonetheless, despite this added structure, we show that  the expectations of \citeauthor{FV04} are vindicated in that the long-term behaviour of the replicator dynamics remains disarmingly simple: (almost) every replicator orbit converges to a Wardrop flow and not merely to the {\em set} of such flows (Theorem \ref{thm:detconvergence}).

\smallskip

Having said that, the imitation procedure inherent in the replicator dynamics implicitly presumes itself that users have perfectly accurate information at their disposal. Unfortunately however, this assumption is not very realistic in networks which exhibit wild delay fluctuations as the result of interference by random exogenous factors (commonly gathered under the collective moniker ``nature''). In population biology, these disturbances are usually modelled by introducing ``aggregate shocks'' to the replicator dynamics \citep{FH92} and, as one would expect, these shocks complicate the situation considerably. For instance, \cite{Ca00} proved that dominated strategies become extinct in the long run, but only if the variance of the shocks is mild enough compared to the payoffs of the game. More recently, \cite{Im05} showed that even equilibrial play arises over time but, again, conditionally on the noise processes not being too loud \citep[see also][]{BHS08, Im09}. On the other hand, if one interprets the replicator dynamics as the derivative of an {\em exponential learning} procedure and perturbs them accordingly (i.e. not as an evolutionary birth-death process), it was shown that similar rationality properties continue to hold, no matter how loud the noise becomes \citep{GameNets09, MM09}.

All the same, these approaches have chiefly focused on Nash-type games where payoffs are multilinear functions over a product of simplices; for example, payoffs in single-population evolutionary games are determined by the bilinear form which is associated to the matrix of the game. This linear structure simplifies things considerably but, unfortunately, congestion models rarely adhere to it; additionally, the notions of Nash and Wardrop equilibrium are at variance in many occasions, a disparity which also calls for a different approach; and, finally, the way that stochastic fluctuations propagate to the users' choices in a network leads to a new stochastic version of the replicator dynamics where the noise processes are no longer independent across users (different paths might share a common subset of links over which disturbances are strongly correlated). On that account, the effect of stochastic fluctuations in congestion models cannot be understood by simply translating previous work on the stochastic replicator dynamics.

\subsection{Outline}

In this paper, we study the distribution of traffic in networks whose links are subject to constant stochastic perturbations that randomly affect the delays experienced by individual traffic elements. This model is presented in detail in Section \ref{sec:prelims}, where we also develop our game-theoretic machinery: specifically, we introduce the notion of a network's {\em redundancy} in Section \ref{subsec:flows}, and we examine its connection to Wardrop equilibria in Section \ref{subsec:Wardrop}. We then derive the rationality properties of the deterministic replicator dynamics in Section \ref{sec:deterministic}, where we show that (almost) every solution trajectory converges to a Wardrop equilibrium.

Section \ref{sec:stochastic} is devoted to the stochastic considerations which constitute the core of our paper. Our first result is that {\em strict} Wardrop equilibria remain stochastically asymptoticaly stable {\em irrespective} of the fluctuations' magnitude (Theorem \ref{thm:stability}); in fact, if the users are ``patient enough'', we are able to estimate the average time it takes them to hit a neighbourhood of the equilibrium in question (Theorem \ref{thm:timestrict}). In conjunction with stochastic stability, this allows us to conclude that when a strict equilibrium exists, users converge to it almost surely (Corollary \ref{cor:stoconvergence}). On the other hand, given that such equilibria do not always exist, we also prove that the replicator dynamics in irreducible networks are recurrent (again under the assumption that the users are patient enough), and we use this fact to show that the long-term average of their traffic distributions concentrates mass in the neighbourhood of an interior Wardrop equilibrium (Theorem \ref{thm:recurrence}).


\subsection{Notational Conventions}
\label{subsec:notation}

If $\set = \{s_\alpha\}_{\alpha=0}^{n}$ is a finite set, the vector space spanned by $\set$ over $\R$ is defined to be the set of all formal linear combinations of elements of $\set$ with real coefficients, i.e. the set of all functions $x:\set\to\R$. In tune with standard set-theoretic notation, we will denote this space by $\R^{\set} \equiv \text{Maps}(\set,\R)$. In this way, $\R^{\set}$ admits a canonical basis $\{e_{\alpha}\}_{\alpha=0}^{n}$ consisting of the indicator functions $e_{\alpha}:\set\to\R$ which take the value $e_{\alpha}(s_{\alpha})=1$ on $s_{\alpha}$ and vanish otherwise; in particular, if $x\in\R^{\set}$ has $x(s_{\alpha}) = x_{\alpha}$, we will have $x = \sum_{\alpha}x_{\alpha}e_{\alpha}$. Hence, under the natural identification $s_{\alpha}\mapsto e_{\alpha}$, we will make no distinction between the elements $s_{\alpha}$ of $\set$ and the corresponding basis vectors $e_{\alpha}$ of $\R^{\set}$ \textendash~in fact, to avoid drowning in a morass of indices, we will routinely use $\alpha$ to refer interchangeably to either $s_{\alpha}$ or $e_{\alpha}$, writing e.g. ``$\alpha\in \set$'' instead of ``$s_{\alpha}\in\set$''. In the same vein, we will also identify the set $\simplex(\set)$ of probability measures on $\set$ with the standard $n$-dimensional simplex of $\R^{\set}$: $\simplex(\set) = \{x\in \R^{\set}: \sum_{\alpha} x_{\alpha} =1 \text{ and }x_{\alpha}\geq 0\}$.

Concerning players and their strategies, we will follow the original convention of \cite{Na51} and employ Latin indices ($i,j,\dotsc$) for players while reserving Greek ones ($\alpha,\beta\dotsc$) for their (pure) strategies; also, to differentiate between strategies, we will use $\alpha,\beta,\ldots$ for indices that start at $0$ and $\mu,\nu,\ldots$ for those that start at $1$. Moreover, if the players' action sets $\act_{i}$ are disjoint (as is typically the case), we will identify their union $\bigcup_{i}\act_{i}$ with their \emph{disjoint union} $\act\equiv\coprod_{i}\act_{i}= \bigcup_{i}\big\{(\alpha,i):\alpha\in\act_{i}\big\}$ by mapping $\alpha\in\act_{i}\mapsto(\alpha,i)\in\act$. Hence, if $\{e_{i\alpha}\}$ is the natural basis of $\R^{\act_{i}}$ and $\{e_{\alpha}\}$ is the corresponding basis of $\R^{\act}\cong\prod_{i}\R^{\act_{i}}$, we will occasionally drop the index $i$ altogether and write $x=\sum_{\alpha} x_{\alpha} e_{\alpha}\in\R^{\act}$ instead of $x=\sum_{i,\alpha} x_{i\alpha} e_{i\alpha}\in\prod_{i}\R^{\act_{i}}$. Similarly, when it is clear from the context that we are summing over the strategy set $\act_{i}$ of player $i$, we will use the shorthand $\sum_{\alpha}^{i} \equiv \sum_{\alpha\in\act_{i}}$.

Finally, if $X(t)$ is some stochastic process in $\R^{n}$ starting at $X(0) = x$ and there is no doubt that we are referring to the process $X$, its law will be denoted by $\prob_{x}$. In that case, we will also employ the term ``almost surely'' instead of the somewhat unwieldy ``$\prob_{x}$-almost surely''.

\section{Preliminaries}
\label{sec:prelims}


\subsection{Games in Normal Form}
\label{subsec:games}

Our starting point for the definition of a game in normal form will be a set of {\em players} $\play$, together with a finite measure $\nu$ on $\play$ which ``accounts'' for all players $i\in\play$ (in the sense that the singletons $\{i\}\subseteq\play$ are all $\nu$-measurable).

The players' possible actions in the game will then be represented by their {\em strategy sets} $\strat_{i}$, $i\in\play$. For our purposes, we will assume that these sets are locally compact Hausdorff spaces and that the relative topologies induced on $\strat_{i}\cap\strat_{j}$ agree for all $i,j\in\play$. Thanks to this compatibility conditon, $\strat_{0}\equiv\bigcup_{j}\strat_{j}$ inherits a natural Borel structure arising from the union topology (the finest topology in which the inclusions $\strat_{i}\injects\strat_{0}$ are continuous) and, in this way, an admissible {\em strategy profile} $x\in\prod_{i}\strat_{i}$ will just be a measurable function $x:\play\to\strat_{0}$ which maps $i\mapsto x_{i}\in\strat_{i}$ for all players $i\in\play$. For technical reasons, we will also require that the push-forward measure $x_{*}\nu$ induced on $\strat_{0}$ by $x$ (given by $x_{*}\nu(U) = \nu(x^{-1}(U))$ for any Borel $U\subseteq\strat_{0}$) be inner regular, and, hence, Radon (since $\nu$ is finite).

As is customary, we will identify two profiles which agree $\nu$-almost everywhere, except when we need to focus on the strategy of a particular player $i\in\play$ against that of his {\em opponents} $\play_{-i}\equiv\play\exclude\{i\}$; in that case, we will use the shorthand $(x_{-i};q_{i})$ to denote the profile which agrees with $x$ on $\play_{-i}$ ($\nu$-a.e.) and maps $i\mapsto q_{i}\in\strat_{i}$. The set $\strat$ of all such profiles $x\in\prod_{i}\strat_{i}$ will then be referred to as the {\em strategy space} of the game and is itself a Borel space because it inherits the subspace topology from the product $\prod_{i}\strat_{i}$.

Bearing all this in mind, the fitness of the players' strategic choices will be determined by their {\em payoff functions} (or {\em utilities}) $u_{i}:\strat\to\R$, $i\in\play$; in particular, $u_{i}(x)\equiv u_{i}(x_{-i};x_{i})$ will simply represent the reward that player $i\in\play$ receives in the strategy profile $x\equiv(x_{-i};x_{i})\in\strat$, i.e. when he plays $x_{i}\in\strat_{i}$ against his opponents' strategy $x_{-i}\in\prod_{j\neq i}\strat_{j}$. The only further assumptions that we will make is that these payoff functions be (Borel) measurable and that $u_{i}(x_{-i};x_{i}) = u_{i}(x_{-i}';x_{i})$ whenever $x$ and $x'$ agree $\nu$-a.e. on $\play_{-i}$.

\smallskip

This collection of {\em players} $i\in\play$, their {\em strategy sets} $\strat_{i}$, and their {\em payoff functions} $u_{i}:\strat\to\R$ will be our working definition for a {\em game in normal form}, usually denoted by $\game\equiv\game\big(\play,\strat,u\big)$. Additionally, if the payoff functions $u_{i}:\strat\to\R$ happen to be continuous, the game $\game$ will be called {\em continuous} as well.

\smallskip

Needless to say, this abstract definition might appear somewhat opaque, so we will immediately proceed with a few important examples to clarify the concept.

\subsubsection{$N$-person Games}

As the name suggests, the players here are indexed by the finite set $\play = \{1,2,\ldots N\}$ (endowed with the usual counting measure) and the game's strategy space will be the finite product $\strat\equiv\prod_{i}\strat_{i}$ (thus doing away with some of the technical subtleties present in the more general definition).

This point is where we recover the original scenario of \cite{Na51}. To see how, assume that every player $i\in\play$ comes with a finite set $\act_{i}$ of {\em actions} (or {\em pure strategies}) which can be ``mixed'' according to some probability distribution $x_{i}\in\simplex(\act_{i})$. In this interpretation, the players' strategy sets are just the simplices $\strat_{i}\equiv\simplex(\act_{i})$ and their payoff functions $u_{i}:\strat\equiv\prod_{i}\strat_{i}\to\R$ are given by the multilinear expectations:
\begin{equation}
u_{i}(x) = u_{i}(x_{1},\ldots x_{N})
= \sum_{\alpha_{1}\in\act_{1}} \negspace\cdots\negspace \sum_{\alpha_{N}\in\act_{N}} x_{1,\alpha_{1}}\cdots x_{N,\alpha_{N}} u_{i,\alpha_{1}\ldots \alpha_{N}},
\end{equation}
where $x_{i}=\sum_{\alpha}^{i} x_{i\alpha} e_{i\alpha}$ in the standard basis $\{e_{i\alpha}\}$ of $\R^{\act_{i}}$ and $u_{i,\alpha_{1}\ldots\alpha_{N}}$ is the reward that player $i$ would obtain by choosing $\alpha_{i}\in\act_{i}$ against his opponents' action $\alpha_{-i}\in\act_{-i}\equiv\prod_{j\neq i}\act_{j}$. Because of this (multi)linear structure, we will commonly refer to Nash-type games as {\em linear games} to contrast them with more general $N$-person games where payoffs and strategy sets might fail to have any sort of linear structure \textendash~as is the case for example with {\em concave games} \citep{Ro65}.

\subsubsection{Population Games}

The cornerstone of evolutionary game theory concerns games played by an uncountable number of players \textendash~for instance, see \cite{Sc73}. As such, these nonatomic {\em population games} require the full breadth afforded by our more abstract definition.

\smallskip

The first piece of additional structure encountered in these games is a measurable partition $\play=\bigcup_{r=1}^{N}\play_{r}$ of the player set $\play$ into $N$ disjoint {\em populations} (or {\em classes}) $\play_{r}\subseteq\play$; accordingly, every player $i\in\play$ belongs to a unique class $\play_{r}$ which we denote by $\class(i)$. Each of these populations is then ``measured'' by the corresponding restriction $\nu_{r}$ of the measure $\nu$ on $\play_{r}$ (i.e. $\nu_{r}(B) = \nu(B\cap\play_{r})$ for any Borel $B\subseteq\play$), and the basic underlying assumption is that these measures are nonatomic.

The second fundamental assumption is that this classification of players also determines how they interact with their environment and with each other. More precisely, this means that the strategy sets of two players that belong to the same population coincide: $\strat_{i}=\strat_{j}$ whenever $\class(i) = \class(j)$. Because of this, we will write $\act_{r}$ for the common strategy set of the $r$-th population and $\act_{0}$ for the corresponding union: $\act_{0}=\bigcup_{r=1}^{N}\act_{r}=\bigcup_{i\in\play}\strat_{i}$.

\vskip2pt

Now, every strategy profile $x:\play\to\act_{0}$ pushes forward a (Radon) measure $\hat x_{r}$ on $\act_{r}$ in the usual way:
\begin{equation}
\hat x_{r}(U)\equiv (x_{*}\nu_{r})(U)=\nu_{r}(x^{-1}(U)) = \nu\{i\in\play_{r}: x_{i}\in U\}
\end{equation}
for any Borel $U\subseteq\act_{r}$ \textendash\ in other words, $\hat x_{r}(U)$ is just the measure of the players in the $r$-th population whose chosen strategy lies in $U\subseteq\act_{r}$. Then, the final (and perhaps most significant) requirement in population games is that the players' payoffs depend only on the {\em strategy distribution} $\hat x = (\hat x_{1},\ldots \hat x_{N})$ and not on the players' individual strategic choices. Specifically, if $P_{0}(\act)$ denotes the space of all such strategy distributions equipped with the topology of vague convergence, we assume that there exist continuous functions $\hat u_{r}:P_{0}(\act)\times\act_{r}\to\R$, $r=1,\ldots N$, such that:
\begin{equation}
u_{i}(x) = \hat u_{r}(\hat x; x_{i}) \text{ for all $i\in\play_{r}$}.
\end{equation}
Consequently, as long as the overall strategy distribution $\hat x$ stays the same, payoffs remain unaffected even by positive-mass migrations of players from one strategy to another (and not only by migrations of measure zero).

\smallskip

Again, it would serve well to illustrate this abstract definition by means of a more concrete example. To wit, in evolutionary game theory, populations are usually represented by the intervals $\play_{r}=[0,m_{r}]$ where $m_{r}>0$ denotes the ``mass'' of the population under Lebesgue measure. The strategy spaces $\act_{r}$ are typically assumed to be finite, so that a strategy distribution is simply a point in the (finite-dimensional) product of simplices $\prod_{r}m_{r}\simplex(\act_{r})$. Hence, if player $i\in\play_{r}$ picks the strategy $\alpha\in\act_{r}$, his payoff will be given by
\begin{equation}
u_{r\alpha}(x) \equiv u_{r}(x;\alpha),
\end{equation}
where, in a slight abuse of notation, we removed the hats from $\hat u_{r}$ and $\hat x$ in order to stress that they are the fundamental quantities that describe the game (it will always be clear from the context whether we are referring to the distribution $\hat x\in P_{0}(\act)$ or to the actual strategy profile $x:\play\to\act_{0}$).

This choice of notation is very suggestive for another reason as well: if we set $\strat_{r} \equiv m_{r}\simplex(\act_{r})$, then these simplices may be taken as the strategy sets of an associated $N$-person game whose players are indexed by $r=1,2\ldots N$ (that is, they correspond to the populations themselves). The only thing needed to complete this description is to define the payoff functions $u_{r}:\strat\equiv\prod_{r}\strat_{r}\to\R$ in this picture, and a natural choice would be to take the {\em population averages}:
\begin{equation}
u_{r}(x) = \frac{1}{m_{r}} \insum^{r}_{\alpha} x_{r\alpha} u_{r\alpha}(x),
\end{equation}
where $x_{r\alpha}$ are the coordinates of $x$ in $\strat$. However, it is worth keeping in mind that, depending on the situation at hand, this need not be the only reasonable choice for a payoff function (we will explore this issue further in the next section).

\paragraph{Potential Games}

An important subclass of population games arises when the payoffs $u_{r\alpha}$ satisfy the closedness condition:
\begin{equation}
\label{eq:closed}
\frac{\pd u_{r\alpha}}{\pd x_{s\beta}} = \frac{\pd u_{s\beta}}{\pd x_{r\alpha}}
\text{ for all populations $r,s$ and for all strategies $\alpha\in\act_{r},\beta\in\act_{s}$.}
\end{equation}
This condition is commonly referred to as ``externality symmetry'' \citep{Sa01} and it describes games where a marginal increase in the population of players using strategy $\alpha$ has the same effect on the payoffs to players playing strategy $\beta$ as the converse increase. Clearly, since the strategy distributions of these games live in the simply connected polytope $\strat=\prod_{r}\strat_{r}$, condition (\ref{eq:closed}) amounts to the existence of a {\em potential function} $F:\strat\to\R$ such that:
\begin{equation}
\label{eq:potential}
u_{r\alpha}(x) = - \frac{\pd F}{\pd x_{r\alpha}}.
\end{equation}
Hence, if a player $i\in\play_{r}$ makes the switch $\alpha\to\beta$, his payoff will change by:
\begin{equation}
u_{r\beta}(x) - u_{r\alpha}(x)
= -\left(\frac{\pd F}{\pd x_{r\beta}} - \frac{\pd F}{\pd x_{r\alpha}}\right)
=-dF(e_{r\beta}-e_{r\alpha}),
\end{equation}
where $\{e_{r\beta}\}$ denotes the standard basis of $\prod_{r}\R^{\act_{r}}$. In other words, the strategy migration $\alpha\to\beta$ is profitable to a player iff the direction $e_{r\beta}-e_{r\alpha}$ descends the potential $F$. This property of potential games will be extremely important for our purposes and its ramifications underlie a large part of our work.

\subsubsection{Nash Equilibrium and Wardrop's Principle}

Under the umbrella of rationality, selfish players will seek to play those strategies which deliver the best rewards against the choices of their opponents. This leads to the celebrated notion of a \emph{Nash equilibrium}, i.e. a strategy profile $q$ which discourages unilateral deviations:
\begin{equation}
\label{eq:NEQ}
\tag{NEQ}
u_{i}(q) \geq u_{i}(q_{-i};q_{i}')\quad\text{for almost every $i\in\play$ and all strategies $q_{i}'\in\strat_{i}$}
\end{equation}
(see also \cite{Sc73} or \cite{Mi00}).

The seminal result of \cite{Na51} was that $N$-person linear games always possess equilibria of this kind. \cite{Ro65} subsequently extended this result to the class of {\em concave} games (continuous concave payoffs over convex strategy sets), while \cite{Sc73} essentially settled the issue for population games with finite strategy sets \cite[see also][]{AK86}.

In this last instance, Nash equilibria are aptly captured by {\em Wardrop's principle}:
\begin{equation}
\label{eq:Wardrop}
u_{r\alpha}(q) \geq u_{r\beta}(q) \text{ for all $\alpha,\beta\in\act_{r}$ s.t. $q$ assigns positive mass to $\alpha$.}
\end{equation}
To see this, note that if $\alpha\in\act_{r}$ has positive measure in the strategy distribution $q$, then there exists a player $i\in\play_{r}$ (actually a positive mass of such players) with $q_{i}=\alpha$ and such that (\ref{eq:NEQ}) holds. Hence, for every $\beta\in\act_{r}$, we immediately get:
\begin{equation}
u_{r\alpha}(q)
= u_{i}(q_{-i};\alpha)
\geq u_{i}(q_{-i};\beta)
= u_{r\beta}(q).
\end{equation}

If the game in question is also a potential one, we have seen that beneficial migrations descend the potential function, so the minima of the potential correspond to strategy distributions where no unilateral improvement is possible. In fact, the Kuhn-Tucker conditions for the game's potential coincide precisely with the Wardrop characterisation (\ref{eq:Wardrop}) and, hence, the game's equilibria will be the critical points of the potential \cite[Proposition 3.1]{Sa01}.

On account of the above, the equilibrium characterisation (\ref{eq:Wardrop}) will be central in our analysis, so we will examine it in depth in the sections that follow. En passant, we only note here that a similar condition can be laid down for population games with continuous strategy sets. This case has recently attracted quite a bit of interest, but since we will not need this added generality, we will not press the issue further~\textendash~ see instead \cite{Cr05} or \cite{HOR09}.

\subsection{Networks and Flows}
\label{subsec:flows}

Stated somewhat informally, our chief interest lies in networks whose nodes produce traffic that seeks to reach its destination as quickly as possible. However, since the time taken to traverse a path in a network increases as the network becomes congested, it is hardly an easy task to pick the ``path of least resistance'' \textendash~especially given that users compete against each other in their endeavours. As a result, the game-theoretic setup of the previous section turns out to be remarkably appropriate for the analysis of these traffic flows.

\smallskip

Following \cite{RT02, RT04}, let $\graph\equiv\graph(\nodes,\edges)$ be a (finite) directed graph with node set $\nodes$ and edge set $\edges$, and let $\pair=(v,w)$ be an \emph{origin-destination} pair in $\graph$ (i.e. an ordered pair of nodes $v,w\in\nodes$ that can be joined by a path in $\graph$). Suppose further that the origin $v$ of $\pair$ outputs traffic towards the destination node $w$ at some rate $\rho>0$; then, the pair $\sigma$ together with the rate $\rho$ will be referred to as a \emph{user} of $\graph$. In this way, a \emph{network} $\net\equiv\net(\play,\act)$ in $\graph$ will consist of a set of \emph{users} $\play$ (indexed by $i=1,\ldots N$), together with an associated collection $\act\equiv\coprod_{i}\act_{i}$ of sets of paths (or \emph{routes}) $\act_{i}=\{\alpha_{i,0},\alpha_{i,1}\ldots\}$ joining $v_{i}$ to $w_{i}$ (where $\sigma_{i} = (v_{i},w_{i})$ is the origin-destination pair of user $i\in\play$).

Two remarks of a book-keeping nature are now in order: first, since we will only be interested in users with at least a modicum of choice on how to route their traffic, we will take $|\act_{i}|\geq 2$ for all $i$. Secondly, we will be assuming that the origin-destination pairs of distinct users are themselves distinct. Fortunately, neither assumption is crucial: if there is only one route available to user $i$, the traffic rate $\rho_{i}$ can be considered as a constant load on the route; and if two users $i,j\in\play$ with rates $\rho_{i}, \rho_{j}$ share the same origin-destination pair, we will replace them by a single user with rate $\rho_{i}+\rho_{j}$ (see also Section \ref{subsec:Wardrop}). This means that the sets $\act_{i}$ can be assumed disjoint and, as a pleasant byproduct, the path index $\alpha\in\act_{i}$ fully characterizes the user $i$ to whom it belongs \textendash~cf. the conventions of Section \ref{subsec:notation}.

So, if $x_{i\alpha}\equiv x_{\alpha}$ denotes the amount of traffic that user $i$ routes via the path $\alpha\in\act_{i}$, the corresponding traffic flow may be represented as $x_{i} = \sum_{\alpha}^{i} x_{i\alpha} e_{i\alpha}$, where $\{e_{i\alpha}\}$ is the standard basis of the space $V_{i}\equiv\R^{\act_{i}}$. However, for such a flow to be admissible, we must also have $x_{i\alpha}\geq 0$ and $\sum^{i}_{\alpha} x_{i\alpha} = \rho_{i}$; hence, the set of admissible flows for user $i$ will be the simplex $\strat_{i} \equiv \rho_{i} \simplex (\act_{i}) = \big\{x_{i}\in V_{i}: x_{i\alpha}\geq 0\text{ and }\sum^{i}_{\alpha} x_{i\alpha} = \rho_{i} \big\}$. Then, by collecting all these individuals flows in a single profile, a \emph{flow in the network} $\net$ will simply be a point $x=\sum_{i}x_{i}\in\strat\equiv\prod_{i}\strat_{i}$.

An alternative (and very useful!) description of a flow $x\in\strat$ can be obtained by looking at the traffic load that the flow induces on the edges of the network, i.e. at the amount of traffic $y_{r}$ that circulates in each edge $r\in\edges$ of $\graph$. In particular:
\begin{equation}
\label{eq:load}
y_{r} = \insum_{i} y_{ir}
= \insum_{i} \insum^{i}_{\alpha\ni r} x_{i\alpha}
\end{equation}
where $y_{ir}=\sum^{i}_{\alpha\ni r} x_{i\alpha}$ is the load induced on $r\in\edges$ by the individual flow $x_{i}\in\strat$. In this manner, a very important question that arises is whether these two descriptions are equivalent; put differently, whether one can recover the flow distribution $x\in\strat$ from the loads $y_{r}$ on the edges of the network.

To answer this question, let $\{\eps_{r}\}$ be the standard basis of the space $W\equiv\R^{\edges}$ spanned by the edges $\edges$ of $\graph$ and consider the \emph{indicator map} $P^{i}:V_{i}\to W$ which sends a path $\alpha\in\act_{i}$ to the sum of its constituent edges: $P^{i}(e_{i\alpha}) = \sum_{r\in\alpha} \eps_{r}$; obviously, if we set $P^{i}(e_{i\alpha}) = \sum_{r} P^{i}_{r\alpha} \eps_{r}$, we see that the entries of $P^{i}$ will be $P^{i}_{r\alpha} = 1$ if $r\in\alpha$ and $0$ otherwise. We can then aggregate this construction over all $i\in\play$ by considering the product space $V\equiv \R^{\act}\cong \prod_{i} V_{i}$ and the corresponding indicator matrix $P = P^{1}\oplus\cdots\oplus P^{N}$ whose entries take the value $P_{r\alpha}=1$ if the path $\alpha\in\act$ employs the edge $r$ and vanish otherwise. By doing just that, (\ref{eq:load}) takes the simpler form $y_{r} = \sum_{\alpha} P_{r\alpha} x_{\alpha}$ or, even more succinctly, $y=P(x)$. Therefore, the question of whether a flow can be recovered from a load profile can be answered in the positive if the indicator map $P:V\to W$ is injective.

This, however, is not the end of the matter because the individual flows $x_{i}\in\strat_{i}$ actually live in the affine subspaces $p_{i}+Z_{i}$ where $p_{i} = \frac{\rho_{i}}{|\act_{i}|}\sum^{i}_{\alpha}e_{i\alpha}$ is the barycentre of $\strat_{i}$ and $Z_{i} \equiv T_{p_{i}}\strat_{i} = \{z_{i}\in V_{i}: \sum^{i}_{\alpha} z_{i\alpha}=0\}$ is the tangent space to $\strat_{i}$ at $p_{i}$ \textendash~it is also worth keeping in mind that if we set $\act_{i}^{*} = \act_{i}\exclude\{\alpha_{i,0}\}$, then $Z_{i}\cong \R^{\act_{i}^{*}}$. As a result, what is actually of essence here is the action of $P$ on the subspaces $Z_{i}\leq V_{i}$, i.e. the restriction $Q\equiv P|_{Z}:Z\to W$ of $P$ on the subspace $Z\equiv T_{p}\strat \cong \prod_{i}Z_{i}$, where $p=(p_{1},\ldots p_{N})$ is the barycentre of $\strat$. In this way, any two flows $x,x'\in\strat$ will have $z=x'-x\in Z$ and the respective loads $y,y'\in W$ will satisfy:
\begin{equation}
y' - y = P(x')-P(x) = P(z) = Q(z),
\end{equation}
so that $y'=y$ iff $x'-x\in \ker Q$. Under this light, it becomes clear that a flow $x\in\strat$ can be recovered from the corresponding load profile $y\in W$ if and only if $Q$ is injective. For this reason, the map $Q:Z\to W$ will be called the \emph{redundancy matrix of the network $\net$}, giving rise to:
\begin{definition}
\label{def:redundancy}
Let $\net$ be a network in a graph $\graph$ and let $Q$ be the redundancy matrix of $\net$. The \emph{redundancy} $\red(\net)$ of $\net$ is defined to be:
\begin{equation}
\red(\net) \equiv \dim (\ker Q).
\end{equation}
If $\red(\net)=0$, the network $\net$ will be called \emph{irreducible}; otherwise, $\net$ will be called \emph{reducible}.
\end{definition}

\begin{figure}
\subfigure[raggedright][An irreducible network: $\red(\net)=0$.]{
\label{subfig:irrnet}
\begin{tikzpicture}
[scale=0.9,
nodestyle/.style={circle,draw=black,fill=gray!5, inner sep=1pt},
edgestyle/.style={->},
>=stealth]

\coordinate (A) at (-2.598,0);
\coordinate (B) at (0,1.5);
\coordinate (C) at (0,-1.5);
\coordinate (D) at (2.598,0);

\node (A) at (A) [nodestyle] {.};
\node (B) at (B) [nodestyle] {.};
\node (C) at (C) [nodestyle] {.};
\node (D) at (D) [nodestyle] {.};

\node (ceiling) at (0,2) {};
\node (floor) at (0,-2.5) {};

\node (legend) at (0,-2) {\phantom{$\alpha_{1}$} No linearly dependent paths.\phantom{$\alpha_{1}$}};

\draw[edgestyle] (A) to node[right] {} (B);
\draw[edgestyle] (A) to node[above] {} (C);
\draw[edgestyle] (B) to node[below] {} (D);
\draw[edgestyle] (C) to node[below] {} (B);
\draw[edgestyle] (C) to node[right] {} (D);

\draw[dashed,red,->] (A) .. controls ($0.4*(A)+0.6*(B)+(0,-0.25)$) .. 
node[midway, below] {$\quad\alpha_{1,0}$} (B);
\draw[dashed,red,->] (A) .. controls ($(C)+0.1*(B)$) and ($(C)+0.1*(A)$)..
node[near start, above] {$\alpha_{1,1}$} (B);

\draw[densely dashed,blue,->] (C) .. controls ($(B)+0.1*(D)$) and ($(B)+0.1*(C)$)..
node[near end, below] {$\alpha_{2,1}$} (D);
\draw[densely dashed,blue,->] (C) .. controls ($0.4*(C)+0.6*(D)+(0,0.25)$) ..
node[midway, left] {$\alpha_{2,0}\,\,$} (D);

\end{tikzpicture}
}
\qquad
\subfigure[raggedright][A reducible network: $\red(\net)=1$.]{
\label{subfig:rednet}
\begin{tikzpicture}
[scale=0.9,nodestyle/.style={circle,draw=black,fill=gray!10, inner sep=1pt},
edgestyle/.style={->},
>=stealth]

\coordinate (A) at (-2.598,0);
\coordinate (B) at (0,1.5);
\coordinate (C) at (0,-1.5);
\coordinate (D) at (2.598,0);

\node (A) at (A) [nodestyle] {.};
\node (B) at (B) [nodestyle] {.};
\node (C) at (C) [nodestyle] {.};
\node (D) at (D) [nodestyle] {.};

\node (ceiling) at (0,2) {};
\node (floor) at (0,-2.5) {};

\node (legend) at (0,-2)
{\phantom{l}$\alpha_{1,0} + \alpha_{2,1} + \alpha_{3,1}
= \alpha_{1,1} + \alpha_{2,0} + \alpha_{3,0}$\phantom{p}};

\draw[edgestyle] (A) to node[right] {} (B);
\draw[edgestyle] (A) to node[above] {} (C);
\draw[edgestyle] (B) to node[below] {} (D);
\draw[edgestyle] (C) to node[below] {} (B);
\draw[edgestyle] (C) to node[right] {} (D);

\draw[dashed,red,->] (A) .. controls ($0.4*(A)+0.6*(B)+(0,-0.25)$) .. 
node[midway, below] {$\quad\alpha_{1,0}$} (B);
\draw[dashed,red,->] (A) .. controls ($(C)+0.1*(B)$) and ($(C)+0.1*(A)$)..
node[near start, above] {$\alpha_{1,1}$} (B);

\draw[densely dashed,blue,->] (C) .. controls ($(B)+0.1*(D)$) and ($(B)+0.1*(C)$)..
node[near end, below] {$\alpha_{2,1}\,$} (D);
\draw[densely dashed,blue,->] (C) .. controls ($0.4*(C)+0.6*(D)+(0,0.25)$) ..
node[midway, left] {$\alpha_{2,0}\,\,$} (D);

\draw[densely dashed,olivegreen,->] (A) .. controls ($(B)+(0,0.75)$)..
node[near start, left] {$\alpha_{3,0}\,$} (D);
\draw[densely dashed,olivegreen,->] (A) .. controls ($(C)+(0,-0.75)$)..
node[near end, right] {$\,\,\alpha_{3,1}$} (D);

\end{tikzpicture}
}
\caption{The addition of a user may increase the redundancy of a network.}
\label{fig:network}
\end{figure}
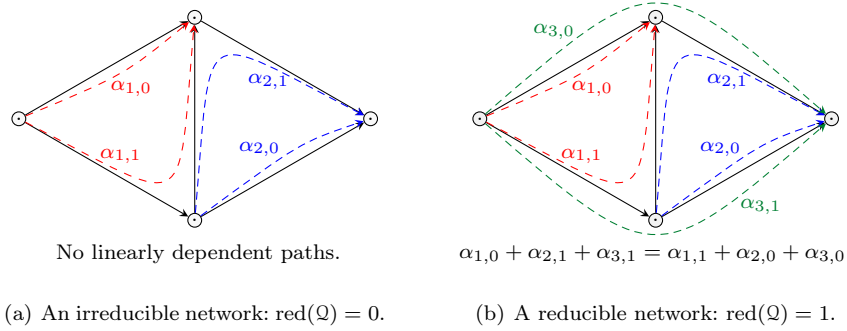

The rationale behind this terminology should be clear enough: when a network $\net$ is reducible, some of its routes are ``linearly dependent'' and the respective directions in $\ker Q$ are ``redundant'' (in the sense that they are not reflected on the edge loads). By comparison, the degrees of freedom of irreducible networks are all active and any statement concerning the network's edges may be translated to one concerning its routes.

This dichotomy between reducible and irreducible networks will be quite significant for our purposes, so it is worth dwelling on Definition \ref{def:redundancy} for a bit more; specifically, it will be important to have a simple recipe with which to compute the redundancy matrix $Q$ of a network $\net$. To that end, let $Q^{i}\equiv P^{i}|_{Z_{i}}$ be the restriction of $P^{i}$ on $Z_{i}$ and, as before, let $\{e_{i,0}, e_{i,1},\ldots\}$ be the standard basis of $V_{i}=\R^{\act_{i}}$. Then, the vectors $\tilde{e}_{i\mu} = e_{i\mu} - e_{i,0}$, $\mu\in\act_{i}^{*}\equiv\act_{i}\exclude\{0\}$, constitute a basis for $Z_{i}$ and it is easy to see that the matrix elements of $Q^{i}$ in this basis will be given by:
\begin{equation}
\label{eq:Qcomponents}
Q^{i}_{r\mu} = P^{i}_{r\mu} - P^{i}_{r,0}.
\end{equation}

The above suggests that if there are too many users in a network, then it is highly unlikely that the network will be irreducible. Indeed, we have:
\begin{proposition}
\label{prop:reducibility}
Let $\net(\play,\act)$ be a network in the graph $\graph(\nodes,\edges)$and let $\edges'\subseteq\edges$ be the set of edges that are present in $\net$. Then:
\begin{equation}
\red(\net) \geq |\play|-|\edges'|.
\end{equation}
Hence, a network will always be reducible if the number of users exceeds the number of available links.
\end{proposition}
\begin{proof}
From the definition of $Q:Z\to W$ we can easily see that $\im Q$ is contained in the subspace of $W$ that is spanned by $\edges'$; furthermore, since every user has $|\act_{i}|\geq 2$ routes to choose from, it follows that $\dim Z = \sum_{i}\left(|\act_{i}| - 1\right) \geq |\play|$. Therefore: $\red(\net) = \dim(\ker Q) = \dim Z - \dim(\im Q) \geq |\play| - |\edges'|$.
\end{proof}

\subsection{Congestion Models and Equilibrium}
\label{subsec:Wardrop}

The time spent by a traffic element on an edge $r\in\edges$ of the graph $\graph$ will be a function $\phi_{r}(y_{r})$ of the traffic load $y_{r}$ on the edge in question -- for example, if the edge represents an M/M/1 queue with capacity $\mu_{r}$, then $\phi_{r}(y_{r}) = 1/(\mu_{r} - y_{r})$. In tune with tradition, we will assume that these \emph{latency} (or \emph{delay}) functions are strictly increasing, and also, to keep things simple, that they are at least $C^{1}$ with $\phi_{r}'>0$.

On that account, the time needed to traverse an entire route $\alpha\in\act_{i}$ will be:
\begin{equation}
\label{eq:delay}
\omega_{i\alpha}(x) = \insum_{r\in\alpha} \phi_{r}(y_{r}) = \insum_{r} P^{i}_{r\alpha} \phi_{r}(y_{r}),
\end{equation}
where as before: $y_{r} = \sum_{\beta} P_{r\beta} x_{\beta}$. In summary, we then have:
\begin{definition}
\label{def:congestion}
A \emph{congestion model} $\model\equiv\model(\net,\phi)$ in a graph $\graph(\nodes,\edges)$ is a network $\net(\play,\act)$ of $\graph$ equipped with a family of increasing latency functions $\phi_{r}, r\in\edges$.
\end{definition}

The similarities between this definition and that of a game in normal form should be evident: all that is needed to turn Definition \ref{def:congestion} into a $N$-person game is to specify its payoff functions. One way to go about this is to consider the \emph{user averages}:
\begin{equation}
\label{eq:avgdelay}
\omega_{i}(x)
= \frac{1}{\rho_{i}} \insum^{i}_{\alpha} x_{i\alpha} \omega_{i\alpha}(x)
= \frac{1}{\rho_{i}}\insum_{r}y_{ir} \phi_{r}(y_{r}),
\end{equation}
where the last equality follows from (\ref{eq:delay}) and the definition of $y_{ir}=\sum^{i}_{\alpha}x_{i\alpha}$. Thus, in keeping with the equilibrium condition (\ref{eq:NEQ}), a flow $q$ will be at Nash equilibrium in the game $\game_{1}\equiv\game_{1}(\play,\strat,-\omega)$ when:
\begin{equation}
\label{eq:avgNash}
\tag{NE1}
\omega_{i}(q) \leq \omega_{i}(q_{-i};q_{i}')\text{ for every user $i\in\play$ and all flows $q_{i}'\in\strat_{i}$}.
\end{equation}
For many classes of latency functions $\phi_{r}$, the average delays $\omega_{i}$ turn out to be convex and the existence of equilibria is assured by the results of \cite{Ro65}. However, not only is this not always the case but, more importantly, the user averages (\ref{eq:avgdelay}) do not necessarily reflect the users' actual optimisation objectives either.

Indeed, another equally justified choice of payoffs is given by the {\em worst delays}:
\begin{equation}
\label{eq:maxdelay}
\widetilde\omega_{i}(x) = \max_{\alpha: x_{i\alpha}>0}\left\{\omega_{i\alpha}(x)\right\},
\end{equation}
i.e. the time at which a user's last traffic packet reaches its destination. In that case, a flow $q$ will be at equilibrium for the game $\game_{2} \equiv \game_{2}(\play,\strat,-\widetilde\omega_{i})$ when:
\begin{equation}
\label{eq:maxNash}
\tag{NE2}
\widetilde\omega_{i}(q) \leq \widetilde\omega_{i}(q_{-i};q_{i}') \text{ for every user $i\in\play$ and all flows $q_{i}'\in\strat_{i}$}.
\end{equation}
Unfortunately, the payoff functions $\widetilde\omega_{i}$ may be discontinuous along any intersection of faces of $\strat_{i}$ because the support $\supp(x_{i}) = \{\alpha\in\act_{i}:x_{i\alpha}>0\}$ of $x_{i}$ changes there as well. Consequently, the existence of equilibrial flows cannot be inferred from the general theory in this instance either.

\smallskip

On the other hand, if we go back to our original motivation (the Internet), we see that our notion of a ``user'' more accurately portrays the network's {\em routers} and not its ``real-life'' users (humans, applications, etc.). However, since these routers are not selfish in themselves, conditions (\ref{eq:avgNash}) and (\ref{eq:maxNash}) do not necessarily point to the right direction either. Instead, the routers' selfless task is to ensure that the {\em nonatomic} traffic elements circulating in the network (the actual selfish entities) remain satisfied. It is thus more reasonable to go back to Wardrop's principle (\ref{eq:Wardrop}):
\begin{definition}
\label{def:Wardrop}
A flow $q\in\strat$ will be at \emph{Wardrop equilibrium} when
\begin{equation}
\label{eq:WEQ}
\tag{WEQ}
\omega_{i\alpha}(q) \leq \omega_{i\beta}(q) \text{ for all $i\in\play$ and for all routes $\alpha,\beta\in\act_{i}$ with $q_{i\alpha}>0$},
\end{equation}
i.e. when every nonatomic traffic element employs the fastest path available to it.
\end{definition}

Condition (\ref{eq:WEQ}) holds as an equality for all routes $\alpha,\beta\in\act_{i}$ that are employed in a Wardrop profile $q$. This gives $\omega_{i}(q) = \omega_{i\alpha}(q)$ for all $\alpha\in\supp(q_{i})$ and leads to the following alternative characterisation of Wardrop flows:
\begin{equation}
\label{eq:WEQ2}
\tag{\ref*{eq:WEQ}$'$}
\omega_{i}(q) \leq \omega_{i\beta}(q) \text{ for all $i\in\play$ and for all $\beta\in\act_{i}$}.
\end{equation}
Even more importantly however, Wardrop equilibria can also be harvested from the (global) minimum of the {\em Rosenthal potential} \citep{Ro73}:
\begin{equation}
\label{eq:Rosenthal}
\Phi(y) = \insum_{r}\Phi_{r}(y_{r}) = \insum_{r} \int_{0}^{y_{r}} \negspace\phi_{r}(w) \dd w.
\end{equation}
The reason for calling this function a potential is twofold: firstly, it is the nonatomic generalisation of the potential function introduced by \cite{MS96} to describe finite congestion games; secondly, the payoff functions $\omega_{i\alpha}$ can be obtained from $\Phi$ by a simple differentiation. To be sure, if we set $F(x) = \Phi(y)$ where $y=P(x)$, we readily obtain:
\begin{equation}
\frac{\pd F}{\pd x_{i\alpha}}
=\insum_{r}\frac{\pd\Phi}{\pd y_{r}} \frac{\pd y_{r}}{\pd x_{i\alpha}}
=\insum_{r} \phi_{r}(y_{r}) P^{i}_{r\alpha}
=\insum_{r\in\alpha} \phi_{r}(y_{r})
=\omega_{i\alpha}(x),
\end{equation}
which is exactly the definition of a potential function in the sense of (\ref{eq:potential}) \textendash~note also that the ``externality symmetry'' condition (\ref{eq:closed}) can be verified independently:
\begin{equation}
\frac{\pd \omega_{i\alpha}}{\pd x_{j\beta}}
=\insum_{r} P^{i}_{r\alpha} \phi_{r}'(y_{r}) P^{j}_{r\beta}
=\insum_{r\in\alpha\cap\beta} \phi_{r}'(y_{r})
=\frac{\pd \omega_{j\beta}}{\pd x_{i\alpha}}.
\end{equation}

To describe the exact relation between Wardrop flows and the minima of $\Phi$, consider the (convex) set $P(\strat)$ of all load profiles $y$ that result from admissible flows $x\in\strat$. Since the latency functions $\phi_{r}$ are in\-crea\-sing, $\Phi$ will be strictly convex over $P(\strat)$ and it will thus have a unique (global) minimum $y^{*}\in P(\strat)$. Amazingly enough, the Kuhn-Tucker conditions that characterise this minimum coincide with the Wardrop condition (\ref{eq:Wardrop}) \citep{BMW56,DS69,Sa01,RT02}, so the Wardrop set of the congestion model $\model$ will be given by:
\begin{equation}
\eq = \{x\in\strat: P(x) = y^{*}\} = P^{-1}(y^{*})\cap\strat.
\end{equation}

\begin{proposition}
Let $\model\equiv\model(\net,\phi_{r})$ be a congestion model with strictly increasing latencies $\phi_{r}$ and let $\eq$ be its set of Wardrop equilibria. Then:
\begin{enumerate}
\item any two Wardrop flows exhibit equal loads and delays.
\item $\eq$ is a nonempty convex polytope with $\dim(\eq)\leq\red(\net)$; moreover, if there exists an interior equilibrium $q\in\Int(\strat)$, then $\dim(\eq)=\red(\net)$.
\end{enumerate}
\end{proposition}

Since $P^{-1}(y^{*})$ is an affine subspace of $\R^{\act}$ and $\strat$ is a product of simplices, there is really nothing left to prove (simply observe that if $q$ is an interior Wardrop flow, then $P^{-1}(y^{*})$ intersects the full-dimensional interior of $\strat$). The only surprise here is that this result seems to have been overlooked in much of the literature concerning congestion models: for instance, both \citet[Corollary 5.6]{Sa01} and \citet[Propositions 2 and 3]{FV04} presume that Wardrop equilibria are unique in networks with increasing latencies. However, if there are two distinct flows $x,x'$ leading to the same load profile $y$ (e.g. as in the simple network of Fig.~\ref{subfig:rednet}), then the potential function $F(x)\equiv \Phi(P(x))$ is no longer strictly convex: it is in fact constant along every null direction of the redundancy matrix $Q=P|_{T\strat}$.

\smallskip

We thus see that a Wardrop equilibrium is unique iff
\begin{inparaenum}[\itshape a\upshape)]
\item the network $\net$ is {\em irreducible}, or
\item $P^{-1}(y^{*})$ only intersects $\strat$ at a vertex.
\end{inparaenum}
This last condition suggests that the vertices of $\strat$ play a special role so, in analogy with Nash games, we define:
\begin{definition}
\label{def:strict}
A Wardrop equilibrium $q$ will be called \emph{strict} if
\begin{inparaenum}[\itshape a\upshape)]
\item $q$ is \emph{pure}: $q = \sum_{i}\rho_{i}e_{i,\alpha_{i}}$, $\alpha_{i}\in\act_{i}$; and
\item$\omega_{i\alpha_{i}}(q) < \omega_{i\beta}(q)$ for all paths $\beta\in\act_{i}\exclude\{\alpha_{i}\}$.
\end{inparaenum}
\end{definition}

In Nash games, a pure equilibrium occasionally fails to be strict, but only by a hair: an arbitrarily small perturbation of a player's pure payoffs $u_{i;\alpha_{1},\ldots\alpha_{N}}$ resolves a pure equilibrium into a strict one without affecting the payoffs of the other players. In congestion models however, there is no such guarantee because, whenever two users' paths overlap, one cannot perturb the delays of one user independently of the other's. As a matter of fact, the existence of a strict Wardrop equilibrium actually {\em precludes} the existence of any other equilibria:
\begin{proposition}
\label{prop:uniquestrict}
Let $\model$ be a congestion model. If $q$ is a strict Wardrop equilibrium of $\model$, then $q$ is the unique Wardrop equilibrium of $\model$.
\end{proposition}
\begin{proof}
Without loss of generality, let $q=\sum_{i}\rho_{i} e_{i,0}$ be a strict Wardrop equilibrium of $\model$ and suppose ad absurdum that $q'\neq q$ is another Wardrop flow. If we set $z=q'-q\in\ker Q$, it follows that the convex combinations $q+\theta z$ will also be Wardrop for all $\theta\in[0,1]$; moreover, for small enough $\theta>0$, $q+\theta z$ employs at least one path $\mu\in\act_{i}\exclude\{0\}$ that is not present in $q$ (recall that $q$ is pure). As a result, we get $\omega_{i\mu}(q+\theta z) = \omega_{i,0}(q+\theta z)$ for all sufficiently small $\theta>0$, and because the latency functions $\omega_{i\alpha}$ are continuous, this yields $\omega_{i,0}(q)=\omega_{i\mu}(q)$. However, since $q$ is a \emph{strict} Wardrop equilibrium which does not employ $\mu$, we must also have $\omega_{i,0}(q) < \omega_{i\mu}(q)$, a contradiction.
\end{proof}

In other words, even if $q$ is a strict equilibrium of a reducible network, then the redundant directions which constitute the affine subspace $q+\ker Q$ will only intersect $\strat$ at $q$. On the other hand, if $q$ is merely a pure equilibrium, $q+\ker Q$ might well intersect the open interior of $\strat$; in that case, there is no arbitrarily small perturbation of the delay functions that could make $q$ into a strict equilibrium.

\paragraph{Equilibria and Objectives}

On account of the above, we will focus our investigations on the concept of Wardrop equilibrium. However, we should mention here that this equilibrial notion can also be reconciled (to some extent at least) with the optimisation objectives represented by the payoffs (\ref{eq:avgdelay}) and (\ref{eq:maxdelay}) as well.

First, with respect to the average delays $\omega_{i}(x) = \rho_{i}^{-1} \sum^{i}_{\alpha}x_{i\alpha}\omega_{i\alpha}(x)$, the {\em optimal} traffic distributions which minimise the aggregate delay $\omega(x) = \sum_{i} \rho_{i} \omega_{i}(x)$ coincide with the Wardrop equilibria of a suitably modified game. This was first noted by \cite{BMW56}, who observed the inherent duality in Wardrop's principle: just as Wardrop equilibria occur at the minimum of the Rosenthal potential, so can one obtain the minimum of the aggregate latency $\omega$ by looking at the Wardrop equilibria of an associated congestion model. More precisely, the only change that needs to be made is to consider the ``marginal'' latency functions $\phi_{r}^{*}(y_{r}) = \phi_{r}(y_{r}) + y_{r}\phi_{r}'(y_{r})$ \cite[see also][]{RT04}. Then, to study these ``socially optimal'' flows, we simply have to redress our analysis to fit these ``marginal latencies'' instead (see Section \ref{sec:discussion} for more details).

\smallskip

Secondly, Wardrop equilibria also have close ties with the Nash condition (\ref{eq:maxNash}) which corresponds to the ``worst-delays'' (\ref{eq:maxdelay}). Specifically, one can easily see that the Nash condition (\ref{eq:maxNash}) is equivalent to the Wardrop condition (\ref{eq:Wardrop}) when every user only has 2 possible paths to choose from (every amount of traffic diverted from one path increases the delay at the user's other path). However, if a user has 3 or more paths at his disposal, then the situation can change dramatically because of {\em Braess's paradox} \citep{Br68}.

The essence of this paradox is that there exist networks which perform better if one removes their fastest link. An example of such a network is given in Fig. \ref{fig:Braess}, where it is assumed that a user seeks to route 6 units of traffic from $A$ to $D$ using the three paths $A\to B\to D$ (blue), $A\to C\to D$ (red) and $A\to B\to C\to D$ (green). In that case, the Wardrop condition (\ref{eq:WEQ}) calls for equidistribution: 2 units are routed via each path, leading to a delay of 92 time units along all paths. Alternatively, if the user sends 3 traffic units via the red and blue paths and ignores the green one, all traffic will experience a delay of 83. Paradoxically, even though the green path has a latency of only 70, the Nash conditon (\ref{eq:maxNash}) is satisfied: if traffic is diverted from, say, the red path to the faster green one, then the latency of the blue path will also increase, thus increasing the worst delay $\widetilde\omega_{i}$ as well.

This paradox is what led to the original investigations in the efficiency of selfish routing \citep{KP99, RT02}, and it seems that it is also what causes this disparity between Wardrop and Nash equilibria. A thorough investigation of this matter is a worthy project but, since it would take us too far afield, we will not pursue it here. Henceforward, we will focus almost exclusively on Wardrop flows, which represent the most relevant equilibrium concept for our purposes.






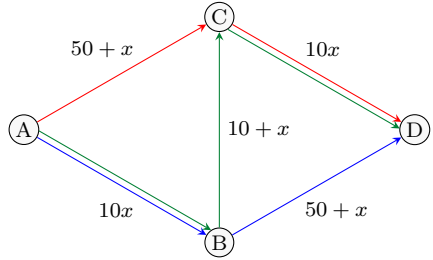
\begin{figure}
\begin{tikzpicture}
[nodestyle/.style={circle,draw=black,fill=gray!5, inner sep=1pt},
edgestyle/.style={->},
>=stealth]

\coordinate (A) at (-2.598,0);
\coordinate (B) at (0,-1.5);
\coordinate (C) at (0,1.5);
\coordinate (D) at (2.598,0);

\node[nodestyle] (A) at (A) {A};
\node[circle, inner sep=4.5pt] (Aup) at ($(A)+(0,0.1)$) {};
\node[nodestyle] (B) at (B) {B};
\node[circle, inner sep=2.4pt] (Bup) at ($(B)+(0,0.1)$) {};
\node[nodestyle] (C) at (C) {C};
\node[circle, inner sep=2.5pt] (Cdown) at ($(C)-(0,0.1)$) {};
\node[nodestyle] (D) at (D) {D};
\node[rectangle, inner sep=5.5pt] (Ddown) at ($(D)-(0,0.1)$) {};

\draw[edgestyle,blue] (A) to node[near end, left, black] {$10x\quad$} (B);
\draw[edgestyle,red] (A) to node[near end, left, black] {$50+x\quad$} (C);
\draw[edgestyle,blue] (B) to node[near start, right, black] {$\quad 50+x$} (D);
\draw[edgestyle,red] (C) to node[near start, right, black] {$\quad 10x$} (D);
\draw[edgestyle,olivegreen] (Aup) to (Bup);
\draw[edgestyle,olivegreen] (B) to node[right, black] {$10+x$} (C);
\draw[->,olivegreen] (Cdown) to (Ddown);

\end{tikzpicture}
\caption{Braess's paradox and the disparity between Wardrop and Nash equilibria.}
\label{fig:Braess}
\end{figure}


\section{Learning, Evolution and Rational Behaviour}
\label{sec:deterministic}

Unfortunately, locating the Wardrop equilibria of a network is a rather arduous process which entails a good deal of global calculations (namely the minimisation of a nonlinear convex functional with exponentially many variables over a convex polytope). Since such calculations clearly exceed the deductive capabilities of individual users (especially if they do not have access to global information), it is of great interest to see whether there are simple learning schemes which allow users to reach an equilibrium without having to rely on centralised computations.

\subsection{Learning and the Replicator Dynamics}
\label{subsec:replicator}

For our purposes, a learning scheme will be a rule which trains users to route their traffic in an efficient way by processing information that is readily available. On the other hand, since this information must be ``local'' in nature, the learning scheme should be similarly ``distributed'': for example, the play of one's opponents or the exact form of the network's latency functions are not easily accessible pieces of information. Furthermore, we should also be looking for a learning scheme which is simple enough for users to apply in real-time, without having to perform a huge number of calculations at each instant.

In continuous time, such a learning scheme may be cast as a dynamical system:
\begin{equation}
\frac{dx}{dt} = v(x) \text{ or, in coordinates: } \frac{d x_{i\alpha}}{dt} = v_{i\alpha}(x),
\end{equation}
where $x(t)\in\strat$ denotes the flow at time $t$ and the vector field $v:\strat\to \R^{\act}$ plays the part of the ``learning rule'' in question \textendash~for simplicity, we will also take $v$ to be smooth. Of course, since the flow $x(t)$ evolves in $\strat$, $v$ itself must lie on the tangent space $Z$ of $\strat$; we thus require that $\sum^{i}_{\alpha} v_{i\alpha}(x) = 0$ for all $i\in\play$.

Furthermore, $v$ should also leave the \emph{faces} of $\strat$ invariant in the sense that any individual trajectory $x_{i}(t)$ that begins at some face of $\strat_{i}$ must always remain in said face. This is actually an essential consequence of our postulates: if a user does not employ a particular route $\alpha\in\act_{i}$, then he has \emph{no} information on the route and, as such, there is no a priori reason that an adaptive learning rule would induce the user to sample it. In effect, such a learning rule would either fail to rely solely on readily observable information or would not necessarily be a very simple one.

This shows that $v_{i\alpha}(x)$ must vanish if $x_{i\alpha}=0$, so if we set $v_{i\alpha}(x) = x_{i\alpha} \tilde{v}_{i\alpha}(x)$, we obtain the orthogonality condition $\insum^{i}_{\alpha} x_{i\alpha} \tilde{v}_{i\alpha}(x)= 0$. Accordingly, $\tilde{v}_{i\alpha}$ may be written in the form:
\begin{equation}
\tilde{v}_{i\alpha}(x) = u_{i\alpha}(x) - u_{i}(x)
\end{equation}
where the $u_{i\alpha}$ satisfy no further constraints and, as can be shown by a simple summation, the function $u_{i}(x)$ is just the \emph{user average}: $u_{i}(x) = \rho_{i}^{-1}\sum^{i}_{\beta} x_{i\beta} u_{i\beta}(x)$ (recall that $\sum^{i}_{\beta} x_{i\beta} = \rho_{i}$). This shows that {\em any} learning rule which leaves the faces of $\strat$ invariant must necessarily be of the form:
\begin{equation}
\label{eq:CRD}
\frac{d x_{i\alpha}}{dt} = x_{i\alpha} \left(u_{i\alpha}(x) - u_{i}(x)\right).
\end{equation}

Dynamics of this type were first derived in the context of population biology by \cite{TJ78}, initially for different genotypes within a species (single-population models), and then for different species altogether (multi-po\-pu\-lation models; \cite{We95} provides an excellent survey). In these evolutionary games, the key objects of interest are large populations of different species, each of them subdivided into distinct {\em genotypes} that are ``programmed'' to a specific behaviour (e.g. ``hawks'' {\sf fight}, while ``doves'' take {\sf flight}). Then, at each instance of biological interaction, it is assumed that one representative from each species is selected at random, and they are all matched to play some Nash game $\game$ whose payoffs represent a proportionate increase in their reproductive fitness (measured by the number of offsprings in the unit of time). In this fashion, if $u_{i\alpha}(x)$ denotes the population average of the payoff to the $\alpha$-th genotype, it turns out that the evolution of the species will be governed by the {\em replicator dynamics} (\ref{eq:CRD}).

\smallskip

In our case, the most natural choice for the payoffs $u_{i\alpha}$ of (\ref{eq:CRD}) is to use the delay functions $\omega_{i\alpha}(x)$ and set $u_{i\alpha}=-\omega_{i\alpha}$. In so doing, we obtain:
\begin{equation}
\label{eq:RD}
\frac{d x_{i\alpha}}{dt}
= x_{i\alpha} \left(\omega_{i}(x) - \omega_{i\alpha}(x)\right).
\end{equation}
In keeping with our ``local information'' mantra, we see that users do not need to know the delays along paths that they do not employ because the replicator vector field vanishes when $x_{i\alpha}=0$. Thus, users that evolve according to (\ref{eq:RD}) are oblivious to their surroundings, even to the existence of other users: they simply use (\ref{eq:RD}) to respond to the stimuli $\omega_{i\alpha}(x)$ in the hope of minimising their delays.

Alternatively, if players learn at different rates $\lambda_{i}>0$ as a result of varied stimulus-response characteristics, we obtain the \emph{rate-adjusted} dynamics:
\begin{equation}
\label{eq:LRD}
\frac{d x_{i\alpha}}{dt}
= \lambda_{i}x_{i\alpha} \left(\omega_{i}(x) - \omega_{i\alpha}(x)\right)
\end{equation}
(naturally, the uniform case (\ref{eq:RD}) is recovered when all players learn at the ``standard'' rate $\lambda_{i}=1$).  Interestingly enough, these learning rates can also be viewed as (player-specific) inverse temperatures: in high temperatures (small $\lambda_{i}$), the differences between routes are toned down and players evolve along the slow time-scales $\lambda_{i}t$; at the other end of the spectrum, if $\lambda_{i}\to\infty$, equation (\ref{eq:LRD}) ``freezes'' to a rigid (and myopic) best-reply process \cite[see also][]{BS97}.

\subsection{Entropy and Rationality}
\label{subsec:entropy}

An immediate observation concerning the replicator dynamics (\ref{eq:LRD}) is that Wardrop equilibria are rest points: if $q$ is a Wardrop flow, the characterisation (\ref{eq:WEQ2}) gives $\omega_{i\alpha}(q) = \omega_{i}(q)$ whenever $x_{i\alpha}>0$. However, the same holds for {\em all} flows $q'$ which exhibit equal latencies along the paths in their support, and these flows are not necessarily Wardrop (in the terminology of \cite{Sa01}, this means that the replicator dynamics are ``complacent''). Consequently, the issue at hand is whether or not the replicator dynamics manage to single out Wardrop equilibria among other stationary states.

In that direction, if $y^{*}$ is the minimum of the Rosenthal potential $\Phi(y)$, it is easy to see that the function $F_{0}(x) = \Phi(P(x)) - \Phi(y^{*})$ is a {\em semi-definite Lyapunov function} for the dynamics (\ref{eq:LRD}). Indeed, $F_{0}$ vanishes on the Wardrop set $\eq$, is positive otherwise, and its evolution under (\ref{eq:LRD}) satisfies:
\begin{equation}
\label{eq:PC}
\frac{d F_{0}}{dt}
= \insum_{i,\alpha} \frac{\pd F_{i\alpha}}{\pd x_{i\alpha}} \frac{d x_{i\alpha}}{dt}
=\insum_{i} \lambda_{i}\rho_{i} \left[\omega_{i}^{2}(x)
- \rho_{i}^{-1}\left(\insum^{i}_{\alpha} x_{i\alpha} \omega_{i\alpha}^{2}(x)\right)\right] \leq 0,
\end{equation}
the last step following from Jensen's inequality \textendash~equality only holds when $\omega_{i\alpha}(x)=\omega_{i}(x)$ for all $\alpha\in\supp(x)$. Thus, by standard results in the theory of dynamical systems, it follows that the solution orbits of (\ref{eq:LRD}) descend the potential $F_{0}$ and eventually converge to a connected subset of rest points \textendash~see also \cite{Sa01}, where the property (\ref{eq:PC}) is referred to as ``positive correlation''.

\smallskip

Nevertheless, since not all stationary points of (\ref{eq:LRD}) are Wardrop equilibria, this result tells us little about the rationality properties of the replicator dynamics in congestion models. A much more important role is played by the {\em relative entropy} (also known as the {\em Kullback-Leibler divergence}):
\begin{equation}
\label{eq:entropy}
H_{q}(x) \equiv \dkl (q,x) = \negspace\sum_{\alpha\in\supp(q)} q_{i\alpha} \log \frac{q_{i\alpha}}{x_{i\alpha}}
\end{equation}
where the sum is taken over the support of $q$: $\supp(q)=\{\alpha: q_{i\alpha}>0\}$. Of course, this sum is finite only when $x$ employs with positive probability all $\alpha\in\act$ that are present in $q$; that is, the domain of definition of $H_{q}$ is $\strat_{q}\equiv\{x\in\strat: q\ll x\}$ (here, ``$\ll$'' denotes absolute continuity of measures). Even so though, it will matter little if we extend $H_{q}$ continuously to all of $\strat$ by setting $H_{q}=\infty$ outside $\strat_{q}$, so we will occasionally act as if $H_{q}$ were defined over all of $\strat$.

\smallskip

Technicalities aside, the significance of the relative entropy lies in that it measures distance in probability space. Indeed, even though it is not a distance function per se (it fails to be symmetric and does not satisfy the triangle inequality), it is positive definite (see below) and strictly convex \citep{We95}. More importantly for our purposes, it is also a (semi-definite) Lyapunov function for the dynamics (\ref{eq:RD}):

\begin{lemma}
\label{lem:entropy}
Let $\model(\net,\{\phi_{r}\})$ be a congestion model with increasing latencies, and let $q\in\strat$ be a Wardrop flow of $\model$. Then, the relative entropy $H_{q}(x)$ satisfies:
\begin{enumerate}
\item $H_{q}(q)=0$ and $H_{q}(x)>0$ for all $x\neq q$;
\item $\dot H_{q}$ vanishes on the Wardrop set $\eq(\model)$ and is negative otherwise (where $\dot H_{q}$ denotes the time derivative with respect to (\ref{eq:RD}): $\dot H_{q} = \sum_{i,\alpha} \frac{\pd H_{q}}{\pd x_{i\alpha}} \, \dot x_{i\alpha}$).
\end{enumerate}
In particular, if the network $\net$ is irreducible ($\red(\net)=0$), then $H_{q}$ is Lyapunov for the replicator dynamics (\ref{eq:RD}).
\end{lemma}

\begin{proof}
The first part of the lemma (positive-definiteness) is an easy consequence of Jensen's inequality \citep[pp. 95--100]{We95}. As for the second part:
\begin{equation}
\label{eq:Hdot}
\dot H_{q}(x)
= \insum_{i,\alpha} \frac{\pd H_{q}}{\pd x_{i\alpha}} \,\dot x_{i\alpha}
= -\insum_{i,\alpha} q_{i\alpha} \left(\omega_{i}(x) - \omega_{i\alpha}(x)\right)
\equiv -L_{q}(x),
\end{equation}
where we have set $L_{q}(x) \equiv \sum^{i}_{\alpha}q_{i\alpha}(\omega_{i}(x) - \omega_{i\alpha}(x))$. Then, the  simple rearrangement $\sum^{i}_{\alpha} q_{i\alpha} \omega_{i}(x) = \rho_{i} \omega_{i}(x) = \sum^{i}_{\alpha} x_{i\alpha} \omega_{i\alpha}(x)$ and some trivial linear algebra yield:
\begin{equation}
\label{eq:adjoint}
L_{q}(x)
\equiv \insum_{i,\alpha}(x_{i\alpha} - q_{i\alpha}) \omega_{i\alpha}(x)
= \insum_{r} (y_{r}-y^{*}_{r}) \phi_{r}(y_{r})
\equiv \Lambda (y),
\end{equation}
where $y=P(x)$ and $y^{*}=P(q)$ are the load profiles which correspond to the flows $x$ and $q$ respectively.

We will refer to the expression $L_{q}$ (or, interchangeably, to $\Lambda$) as the {\em adjoint potential} of $\model$ because, similarly to the Rosenthal potential $\Phi$, it measures distance from the Wardrop set $\eq$. The properties of $L_{q}$ will be discussed at length in Appendix \ref{apx:adjoint} where, among others, we establisth the easy (but crucial!) inequality:
\begin{equation}
\label{eq:potbound}
\Lambda(y) \geq \Phi(y) - \Phi(y^{*}).
\end{equation}
Hence, with $y^{*}$ being the global minimum of $\Phi$, we conclude that $\Lambda$ is positive definite and the lemma follows by noting that $P(x)=y^{*}$ iff $x$ is Wardrop.
\end{proof}

\begin{remark}[\emph{The Rate-adjusted Case}]
It is also reasonable to ask whether the relative entropy function enjoys the same properties in the rate-adjusted dynamics (\ref{eq:LRD}). Unfortunately, this is not true unless all players learn at the same rate; however, if we consider the \emph{rate-adjusted} relative entropy:
\begin{equation}
\label{eq:adjentropy}
H_{q}(x;\lambda)
= \sum_{i\in\play} \lambda_{i}^{-1} \negspace\sum_{\alpha:q_{i\alpha}>0} q_{i\alpha} \log\frac{q_{i\alpha}}{x_{i\alpha}},
\end{equation}
the same calculations show that $\dot H_{q}(x;\lambda) =-L_{q}(x)$ and provide us with the analogue of Lemma \ref{lem:entropy} for the rate-adjusted dynamics (\ref{eq:LRD}).
\end{remark}

\setcounter{remark}{0}

In view of the above, it would be tempting to infer that the replicator dynamics converge to Wardrop equilibrium. Nevertheless, a semi-definite Lyapunov function is not enough to guarantee convergence by itself, even if we rule out the existence of limit cycles. For instance, if we consider the homogeneous system:
\begin{equation}
\label{eq:counterexample}
\dot x = yz,\quad \dot y = -xz,\quad \dot z=-z^{2},
\end{equation}
with $z\geq0$, we see that it admits the semi-definite Lyapunov function $H(x,y,z) = x^{2}+y^{2}+z^{2}$ whose time derivative only vanishes on the $x$-$y$ plane. However, the general solution of (\ref{eq:counterexample}) in cylindrical coordinates $(\rho,\phi,z)$ is just:
\begin{equation}
\label{eq:helix}
\rho(t) = \rho_{0},
\quad \phi(t) = \phi_{0} - \log(1+z_{0}t),
\quad z(t)=\frac{z_{0}}{1+z_{0}t},
\end{equation}
and this represents a helix of constant radius whose coils become topologically dense as the solution orbits approach the $x$-$y$ plane. We thus see that the solutions of (\ref{eq:counterexample}) approach a set of stationary points, but do not converge to a specific one.

\smallskip

That said, there is much more at work in the replicator dynamics (\ref{eq:LRD}) than a single semi-definite Lyapunov function: there exists a whole \emph{family} of such functions, one for each Wardrop flow $q\in\strat$. So, undettered by potential pathologies, the replicator dynamics actually \emph{do} converge to equilibrium:

\begin{theorem}
\label{thm:detconvergence}
Let $\model(\net,\{\phi_{r}\})$ be a congestion model in a network $\net$. Then, every interior solution trajectory of the replicator dynamics (\ref{eq:LRD}) converges to a Wardrop equilibrium of $\model$; in particular, if the network $\net$ is irreducible, $x(t)$ converges to the unique Wardrop equilibrium of $\model$.
\end{theorem}

\begin{proof}
It will be useful to shift our point of view to the {\em evolution function} $\theta(x,t)$ of the dynamics (\ref{eq:LRD}) which describes the solution trajectory that starts at $x$ at time $t=0$ and which satisfies the consistency condition:
\begin{equation}
\theta(x,t+s) = \theta(\theta(x,t),s) \text{ for all $t,s\geq0$ and for all $x\in\strat$.}
\end{equation}
Now, fix the initial condition $x\in\Int(\strat)$ and let $x(t)=\theta(x,t)$ be the corresponding solution orbit. If $q\in\eq$ is a Wardrop equilibrium of $\model$, then, in view of Lemma \ref{lem:entropy}, the function $V_{q}(t) \equiv H_{q}(\theta(x,t))$ will be decreasing and will converge to some $m\geq0$ as $t\to\infty$. It  thus follows that $x(t)$ converges itself to the level set $H^{-1}_{q}(m)$.

Suppose now that there exists some increasing sequence of times $t_{n}\to\infty$ such that $x_{n}\equiv x(t_{n})$ does not converge to $\eq$. By compactness of $\strat$ (and by descending to a subsequence if necessary), we may assume that $x_{n}=\theta(x,t_{n})$ converges to some $x^{*}\notin \eq$ (but necessarily in $H^{-1}_{q}(m)$). Hence, for any $t>0$:
\begin{equation}
H_{q}(\theta(x,t_{n}+t)) = H_{q}(\theta(\theta(x,t_{n}),t)) \to H_{q}(\theta(x^{*},t)) < H_{q}(x^{*}) = m
\end{equation}
where the (strict) inequality stems from the fact that $\dot H_{q}<0$ outside $\eq$. On the other hand, $H_{q}(\theta(x,t_{n}+t))=V_{q}(t_{n}+t)\to m$, a contradiction.

Since the sequence $t_{n}$ was arbitrary, this shows $x(t)$ converges to the set $\eq$. So, let $q'$ be a limit point of $x(t)$ with $x(t'_{n})\to q'$ for some sequence of times $t_{n}'\to\infty$. Then, $V_{q'}(t_{n}')=H_{q'}(x(t_{n}'))$ will converge to zero and, with $V_{q'}$ decreasing, we will have $\lim_{t\to\infty} V_{q'}(t) =0$ as well. Seeing as $H_{q'}$ only vanishes at $q'$, we conclude that $x(t)\to q'$.
\end{proof}

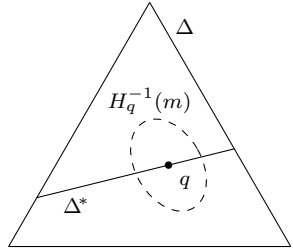
\begin{figure}
\begin{tikzpicture}[scale=1.25]
\coordinate (A) at (-1.5,0);
\coordinate (B) at (0,2.598);
\coordinate (C) at (1.5,0);
\coordinate (P1) at ($ 4/5*(A) + 1/5*(B)$);
\coordinate (P2) at ($ 2/5*(B) + 3/5*(C)$);
\coordinate (Q) at ($1/3*(P1) + 2/3*(P2)$);
\draw (A)--(B);
\draw (B)--(C);
\draw (C)--(A);
\draw (P1)--(P2);
\draw[dashed, rotate=30] (Q) ellipse (10pt and 15pt);
\node [fill=black, shape =circle, inner sep=1pt, label=-45: $q$] at (Q) {};
\node at ($ 0.9*(B) + 0.25*(C)$) {$\strat$};
\node [inner sep=-2pt, label=-90: $\eq$] at ($ 0.8*(P1) + 0.2*(P2)$) {};
\node [inner sep=0pt] at (0,1.55) {$H^{-1}_{q}(m)$};
\end{tikzpicture}
\caption{The various sets in the proof of Theorem \ref{thm:detconvergence}.}
\label{fig:convproof}
\end{figure}

\begin{remark}[{\em Previous Work}]
In the context of potential games, \cite{Sa01} examined a class of learning dynamics $v_{i\alpha}$ which are ``positively correlated'' to the game's payoff functions $u_{i\alpha}=-\omega_{i\alpha}$, in the sense that $\sum_{i,\alpha} v_{i\alpha}(x)\omega_{i\alpha}(x) \geq0$. It was then shown that if the rest points of these dynamics coincide with the game's Wardrop equilibria (the ``non-complacency'' condition), then all solution orbits converge to set of Wardrop equilibria. Unfortunately, as we have already pointed out, the replicator dynamics are ``complacent'' and, in that case, Sandholm's results only ensure that Wardrop equilibria are Lyapunov stable.

To the best of our knowledge, the stronger convergence properties of Theorem \ref{thm:detconvergence} were first suggested by \citet{FV04} who identified the link between Wardrop equilibrium and evolutionary stability \citep{MS74}. In particular, the authors showed that Wardrop equilibria are robust against ``mutations'' that lead to greater delays but, in networks with more than one users (the ``multi-commodity'' case as they call it), their approach rests heavily on the (implicit) assumption of irreducibility. If this is not the case, the adjoint potential $L_{q}$ is only positive \emph{semi-definite} and the approach of \citeauthor{FV04} breaks down because Wardrop equilibria are only \emph{neutrally} stable \textendash~this is also the problem with the formulation of Corrolary 5.1 in \cite{Sa01}.
\end{remark}

\begin{remark}[\emph{Non-interior Trajectories}]
One might also ask what happens if the initial condition $x(0)$ is not an interior point of $\strat$. Clearly, if $x(0)$ does not employ the routes that are present in a Wardrop flow $q$, $x(t)$ cannot have $q$ as a limit point -- a simple consequence of the fact that the replicator dynamics leave the faces of $\strat$ invariant. All the same, one can simply quotient out the routes that are not initially present until $x(0)$ \emph{becomes} an interior point in the reduced strategy space $\strat_{\text{eff}}$ that ensues. In that case, Theorem \ref{thm:detconvergence} can be applied to the (similarly reduced) congestion model $\model_{\text{eff}}$ to show that $x(t)$ converges to Wardrop equilibrium in $\model_{\text{eff}}$ (cf. the ``restricted equilibria'' of \citealp{FV04}).
\end{remark}


\begin{remark}[{\em Evolution and Friction}]
Since the replicator trajectories converge to a Wardrop equilibrium, it follows that there can be no limit cycles. On the other hand, limit cycles are a common occurence in evolutionary games: for example, Mathcing Pennies and Rock-Paper-Scissors both exhibit limit cycles in the standard replicator dynamics \citep{We95}. So, while the evolutionary energy of large populations may remain undiminished over time, Theorem \ref{thm:detconvergence} shows that congestion models are {\em dissipative} and traffic flows settle down to a steady state.
\end{remark}

\setcounter{remark}{0}


\section{The Effect of Stochastic Fluctuations}
\label{sec:stochastic}

Going back to our original discussion on learning schemes, we see that the users' evolution hinges on the feedback that they receive about their choices, namely the delays $\omega_{i\alpha}(x)$ that they record. We have already noted that this information is based on actual observations, but this does not necessarily mean that it is also {\em accurate} as well. For instance, the interference of nature with the game or imperfect readings of one's payoffs might perturb this information considerably; additionally, if the users' traffic flows are not continuous in time but consist of discrete segments instead (e.g. datagrams in communication networks), the queueing latencies $\omega_{i\alpha}$ only represent the users' expected delays. Hence, the delays that users actually observe might only be a randomly fluctuating estimate of the underlying payoffs, and this could negatively affect the rationality properties of the replicator dynamics.

\subsection{Stochastic Replicator Dynamics}
\label{subsec:SRD}

Our goal here will be to determine the behaviour of the replicator dynamics under sto\-chas\-tic perturbations of the kind outlined above. To that end, write the delay that users experience along the edge $r\in\edges$ as $\hat\phi_{r} = \phi_{r}+\eta_{r}$ where $\eta_{r}$ denotes the perturbation process. Then, the latency $\hat\omega_{i\alpha}$ along $\alpha\in\act_{i}$ will just be $\hat \omega_{i\alpha} = \omega_{i\alpha} + \eta_{i\alpha}$, where, in obvious notation, $\eta_{i\alpha} = \sum_{\alpha} P^{i}_{r\alpha} \eta_{r}$. In this way, the replicator dynamics (\ref{eq:RD}) become:
\begin{equation}
\label{eq:perturb}
\frac{d x_{i\alpha}}{dt}
=x_{i\alpha}\left(\hat \omega_{i} - \hat \omega_{i\alpha}\right)
= x_{i\alpha}\left(\omega_{i} - \omega_{i\alpha}\right)
+ x_{i\alpha}(\eta_{i} - \eta_{i\alpha})
\end{equation}
where $\hat\omega_{i} = \rho_{i}^{-1}\sum^{i}_{\beta} x_{i\beta} \hat\omega_{i\beta}$ and $\eta_{i}=\rho_{i}^{-1}\sum^{i}_{\beta}x_{i\beta}\eta_{i\beta}$.

The exact form of the perturbations $\eta_{r}$ clearly depends on the particular situation at hand. Still, since we are chiefly interested in stochastic fluctuations around the underlying delays $\omega_{i\alpha}$, it is reasonable to take these perturbations to be some sort of white noise that does not bias users towards one direction or another. In that case, we should rewrite (\ref{eq:perturb}) as a \emph{stochastic} differential equation:
\begin{equation}
\label{eq:SRD}
dX_{i\alpha}
= X_{i\alpha} \left[\omega_{i}(X) - \omega_{i\alpha}(X)\right] dt
+ X_{i\alpha} \left[ dU_{i\alpha} - \rho_{i}^{-1}\insum^{i}_{\beta}X_{i\beta} \dd U_{i\beta}\right]
\end{equation}
where $dU_{i\alpha}$ describes the total noise along the path $\alpha\in\act_{i}$:
\begin{equation}
\label{eq:U}
dU_{i\alpha}
= \insum_{r\in\alpha} \sigma_{r} \dd W_{r}
= \insum_{r} P^{i}_{r\alpha}\sigma_{r} \dd W_{r}
\end{equation}
and $W(t) = \sum_{r} W_{r}(t) \eps_{r}$ is a Wiener process in $\R^{\edges}$, the space spanned by the edges $\edges$ of the network. Similarly, if players learn at different rates $\lambda_{i}$, we get:
\begin{flalign}
\label{eq:SLRD}
d X_{i\alpha}
&=\lambda_{i} X_{i\alpha} \left[\omega_{i}(X) - \omega_{i\alpha}(X)\right]
+\lambda_{i} X_{i\alpha} \left[ dU_{i\alpha} - \rho_{i}^{-1}\insum^{i}_{\beta} X_{i\beta}\dd U_{i\beta}\right]\\
&= \lambda_{i} b_{i\alpha}(X) dt
+ \lambda_{i}\insum^{i}_{\beta} c_{i,\alpha\beta}(X) \dd U_{i\beta}\notag
\end{flalign}
where $b$ and $c$ are the drift and diffusion coefficients that appear in (\ref{eq:SRD}).

The rate-adjusted equation (\ref{eq:SLRD}) will constitute our stochastic version of the replicator dynamics and, as such, it warrants some discussion in and by itself. A first remark to be made concerns the noise coefficients $\sigma_{r}$: even though we have written them in a form that suggests they are constant, they need not be so: after all, the intensity of the noise on an edge might well depend on the edge loads $Y_{r}=\sum_{\alpha} P_{r\alpha} X_{\alpha}$. On that account, we will only assume that these coefficients are essentially bounded functions of the loads $y$. Nonetheless, in an effort to reduce notational clutter, we will not indicate this dependence explicitly; instead, we simply remark here that our results continue to hold if we replace $\sigma_{r}$ with the worst-case scenario $\sigma_{r}\leftrightarrow\ess\sup_{y} \sigma_{r}(y)$.

Secondly, it is also important to compare (\ref{eq:SLRD}) to other stochastic incarnations of the replicator dynamics, namely the ``ag\-gre\-gate shocks'' version of \cite{FH92} and the authors' own ``exponential learning'' approach \citep{GameNets09,MM09}. In the case of the former, one perturbs the replicator equation (\ref{eq:CRD}) by accounting for the (stochastic) interference of nature with reproduction rates \citep{FH92, Im05}:
\begin{flalign}
\label{eq:ASRD}
dX_{i\alpha}
&= X_{i\alpha}
\left(u_{i\alpha}(X) - u_{i}(X)\right)\dd t
-\left(\sigma_{i\alpha}^{2} X_{i\alpha} - \insum^{i}_{\beta} \sigma_{i\beta}^{2} X_{i\beta}\right)\dd t\\
&+X_{i\alpha}\left[
\sigma_{i\alpha}\dd W_{i\alpha} - \insum^{i}_{\beta} \sigma_{i\beta} X_{i\beta}\dd W_{i\beta}
\right],\notag
\end{flalign}
where $W=\sum_{i,\alpha} W_{i\alpha} e_{i\alpha}$ is a Wiener process in $\prod_{i}\R^{\act_{i}}$. Then, if the ``aggregate shocks'' $\sigma_{i\alpha}$ are mild enough, \cite{Ca00} and \cite{Im05} showed that dominated strategies become extinct and that the game's strict Nash equilibria are asymptotically stable with arbitrarily high probability.

By comparison, in the ``exponential learning'' case it is assumed that the players of a Nash game employ a learning scheme akin to logistic fictitious play \cite[pp.~118--129]{FL98}. However, if the information that players have is imperfect, the errors propagate to their learning curves and instead lead to the stochastic dynamics:
\begin{flalign}
\label{eq:XRD}
dX_{i\alpha}
&= \lambda_{i}X_{i\alpha} \left[
\left(u_{i\alpha}(X) - u_{i}(X)\right)
\right] dt
+\lambda_{i}X_{i\alpha}\!\!\left[
\sigma_{i\alpha}dW_{i\alpha} - \insum^{i}_{\beta} \sigma_{i\beta} X_{i\beta} dW_{i\beta}
\right]\\
&+\frac{\lambda_{i}^{2}}{2}X_{i\alpha}\!\!\left[
\sigma_{i\alpha}^{2} (1-2X_{i\alpha}) - \insum^{i}_{\beta} \sigma_{i\beta}^{2} X_{i\beta}(1-2X_{i\beta})
\right]dt.\notag
\end{flalign}
The rationality properties of these learning dynamics are somewhat stronger than in the biological setting: irrespective of the perturbations' magnitude, strategies which are not rationally admissible die out at an exponential rate and the strict Nash equilibria of the game are always (stochastically) stable \citep{MM09,GameNets09}.

In light of the above, there are two notable traits of (\ref{eq:SLRD}) that set it apart from its other stochastic versions. First off, the drift of (\ref{eq:SLRD}) coincides with the deterministic replicator dynamics (\ref{eq:LRD}) whereas the drift coefficients of (\ref{eq:ASRD}) and (\ref{eq:XRD}) do not. On the other hand, the martingale processes $U$ that appear in (\ref{eq:SLRD}) are not uncorrelated components of some Wiener process (as is the case for both (\ref{eq:ASRD}) and (\ref{eq:XRD})): instead, depending on whether the paths $\alpha,\beta\in\act$ have edges in common or not, the processes $U_{\alpha}, U_{\beta}$ might be highly correlated or not at all.

To make this last observation more precise, recall that the Wiener differentials $\dd W_{r}$ are orthogonal: $dW_{r}\cdot dW_{s} = d[W_{r},W_{s}] = \delta_{rs}\dd t$. In its turn, this implies that the stochastic differentials $dU_{\alpha}, dU_{\beta}$ satisfy:
\begin{flalign}
\label{eq:dUprod}
dU_{\alpha} \ip \dd U_{\beta}
&=\bigg(\insum_{r}P_{r\alpha} \sigma_{r}\dd W_{r}\bigg) \cdot \bigg(\insum_{s}P_{s\beta}\sigma_{s} \dd W_{s}\bigg)\\
&=\insum_{r,s} P_{r\alpha} P_{s\beta} \sigma_{r}\sigma_{s} \delta_{rs} \dd t
=\insum_{r\in\alpha\beta} \sigma_{r}^{2} \dd t = \sigma_{\alpha\beta}^{2} \dd t,\notag
\end{flalign}
where $\sigma_{\alpha\beta}^{2} = \sum_{r} P_{r\alpha} P_{r\beta}\sigma_{r}^{2}$ gives the variance of the noise along the intersection $\alpha\beta\equiv\alpha\cap\beta$ of the paths $\alpha,\beta\in\act$ (note also that we used our notational conventions to avoid cumbersome expressions such as $\sigma_{i\alpha,j\beta}^{2}$). We thus see that the processes $U_{\alpha}$ and $U_{\beta}$ are uncorrelated iff the paths $\alpha,\beta\in\act$ have no common edges. At the other extreme, we have:
\begin{equation}
\left(dU_{\alpha}\right)^{2} = \insum_{r\in\alpha} \sigma_{r}^{2} \dd t = \sigma_{\alpha}^{2} \dd t
\end{equation}
where $\sigma_{\alpha}^{2}\equiv\sigma_{\alpha\alpha}^{2} = \sum_{r}P_{r\alpha}\sigma_{r}^{2}$ measures the intensity of the noise on the route $\alpha\in\act$.\footnote{This notation is also consistent with our intersection notation: $\alpha\alpha \equiv \alpha\cap\alpha = \alpha$.} These expressions will be key to our analysis and we will make liberal use of them in the rest of our paper.

\subsection{Stochastic Fluctuations and Rationality}
\label{subsec:stochrationality}

Our goal in this section will be to explore the rationality properties of the stochastic replicator dynamics (\ref{eq:SLRD}). To begin with, note that (\ref{eq:SLRD}) admits a (unique) strong solution for any initial state $X(0) = x\in\strat$, even though its coefficients do not necessarily grow linearly -- a common requisite for existence and uniqueness of strong solutions to SDE's. Indeed, an addition over $\alpha\in\act_{i}$ reveals that every component simplex $\strat_{i}$ of $\strat$ remains invariant under these dynamics: if $X_{i}(0) = x_{i}\in\strat_{i}$, then $d\left(\sum_{\alpha} X_{i\alpha}\right) = 0$ and, hence, $X_{i}(t)$ stays in $\strat_{i}$ for all $t\geq0$.
So, if $U\supseteq\strat$ is open and $\phi$ is a smooth bump function on $U$ that vanishes outside some compact set $K\supseteq U$, the SDE
\begin{equation}
d X_{i\alpha}
= \lambda_{i}\phi(x)\left(
b_{i\alpha}(X)\dd t
+ \insum^{i}_{\beta} c_{i,\alpha\beta}(X)\dd U_{i\beta}
\right)
\end{equation}
has bounded coefficients and will thus admit a unique strong solution. But since this last equation agrees with (\ref{eq:SLRD}) on $\strat$ and all solutions of (\ref{eq:SLRD}) always stay in $\strat$, our claim follows.

Now, as in the deterministic setting, our main tool will be the (rate-adjusted) relative entropy $H_{q}(x;\lambda) = \sum_{i}\lambda_{i}^{-1}\sum^{i}_{\alpha} q_{i\alpha} \log \left(q_{i\alpha}/x_{i\alpha}\right)$ which we will study with the help of the \emph{generator} $\gen$ of the diffusion (\ref{eq:SLRD}). To that end, recall that the generator $\gen$ of the It\^o diffusion:
\begin{equation}
d X_{\alpha}(t) = \mu_{\alpha}(X(t)) \dd t + \insum_{\beta} \sigma_{\alpha\beta}(X(t)) \dd W_{\beta}(t),
\end{equation}
where $W$ is a Wiener process, is just the second order differential operator:
\begin{equation}
\label{eq:gen}
\gen
= \insum_{\alpha} \mu_{\alpha}(x) \frac{\pd}{\pd x_{\alpha}}
+\frac{1}{2}\insum_{\alpha,\beta} \left(\sigma(x)\sigma^{T}(x)\right)_{\alpha\beta} \frac{\pd^{2}}{\pd x_{\alpha} \pd x_{\beta}}
\end{equation}
(for a comprehensive account, consult the excellent book by \cite{Ok07}).
In this manner, if $f$ is a sufficiently smooth function, $\gen f$ captures the drift of the process $f(X(t))$:
\begin{equation}
\label{eq:gendrift}
d f(X(t))
= \gen f(X(t)) \dd t
+ \insum_{\alpha,\beta} \left.\frac{\pd f}{\pd x_{\alpha}}\right|_{X(t)}
\negspace\sigma_{\alpha\beta}(X(t)) \dd W_{\beta}(t).
\end{equation}

Of course, in the case of the diffusion (\ref{eq:SLRD}), the martingales $U$ are not the components of a Wiener process, so (\ref{eq:gendrift}) cannot be applied right off the shelf. However, a straightforward application of It\^o's lemma (see appendix \ref{apx:stochastic}) yields:
\begin{lemma}
\label{lem:stochentropy}
Let $\gen$ be the generator of (\ref{eq:SLRD}). Then, for any $q\in\strat$:
\begin{flalign}
\label{eq:stochentropy}
\gen H_{q}(x;\lambda)=
-L_{q} (x)
&+\frac{1}{2}\insum_{i}\frac{\lambda_{i}}{\rho_{i}}
\insum^{i}_{\beta,\gamma} \sigma_{\beta\gamma}^{2} (x_{i\beta} - q_{i\beta})(x_{i\gamma} - q_{i\gamma})\\
&+\frac{1}{2}\insum_{i}\frac{\lambda_{i}}{\rho_{i}}
\insum^{i}_{\beta,\gamma} \sigma_{\beta\gamma}^{2} q_{i\beta} (\rho_{i}\delta_{\beta\gamma} - q_{i\gamma})\notag,
\end{flalign}
where $L_{q}(x) = \sum_{i\alpha}(x_{i\alpha} - q_{i\alpha}) \omega_{i\alpha}(x)$ is the adjoint potential of (\ref{eq:adjoint}).
\end{lemma}

In a certain sense, this lemma can be viewed as the stochastic analogue of Lemma \ref{lem:entropy} (which is recovered immediately if we set $\sigma=0$). However, it also shows that the stochastic situation is much more intricate than the deterministic one. For example, if $q$ is a Wardrop equilibrium, (\ref{eq:stochentropy}) gives:
\begin{equation}
\label{eq:positiveLHq}
\gen H_{q}(q;\lambda)
= \frac{1}{2} \insum_{i}\frac{\lambda_{i}}{\rho_{i}}
\insum^{i}_{\beta,\gamma} q_{i\beta} \left(\rho_{i}\delta_{\beta\gamma} - q_{i\gamma}\right) \sigma_{\beta\gamma}^{2},
\end{equation}
and if we focus on user $i\in\play$, we readily obtain:
\begin{flalign}
\label{eq:posinterior}
\insum^{i}_{\beta,\gamma}q_{i\beta} \left(\rho_{i}\delta_{\beta\gamma} - q_{i\gamma}\right)
&\sigma_{\beta\gamma}^{2}
=\insum^{i}_{\beta,\gamma} q_{i\beta}(\rho_{i}\delta_{\beta\gamma} - q_{i\gamma})
\insum_{r}P^{i}_{r\beta}P^{i}_{r\gamma} \sigma_{r}^{2}\\
=\insum_{r}
&\sigma_{r}^{2}\rho_{i}y_{ir} - \insum_{r}\sigma_{r}^{2}y_{ir}^{2} = \insum_{r}\sigma_{r}^{2} y_{ir}(\rho_{i} - y_{ir}),\notag
\end{flalign}
where $y_{ir}=\sum_{r}P^{i}_{r\alpha}q_{i\alpha}\leq \rho_{i}$ is the load induced on edge $r\in\edges$ by the $i$-th user. This shows that (\ref{eq:positiveLHq}) is \emph{positive} if at least one user mixes his routes, thus ruling out negative definiteness (even semi-definiteness) for $\gen H_{q}$.

In view of the above, unconditional convergence to Wardrop equilibrium appears to be a ``bridge too far'' in our stochastic environment, especially when the equilibrium in question is not pure -- after all, mixed equilibria are not even traps of (\ref{eq:SLRD}). This leads us to the notion of \emph{stochastic stability}:
\begin{definition}[\citealp{Ar74,GS71}]
\label{def:stability}
We will say that $q\in\R^{n}$ is \emph{stochastically asymptotically stable} with respect to the process $X(t)$ when, for every neighbourhood $U$ of $q$ and every $\eps>0$, there exists a neighbourhood $V$ of $q$ such that:
\begin{equation}
\label{eq:stability}
\prob_{x}\left\{X(t)\in U \text{ for all } t\geq0 \text{ and }\lim_{t\to\infty} X(t) = q\right\}\geq 1-\eps.
\end{equation}
for all initial conditions $X(0)=x\in V$.
\end{definition}

As we mentioned before, the notion of stochastic stability features pro\-mi\-nent\-ly in the analysis of stochastically perturbed evolutionary or Nash-type games because it is precisely the type of stability that the strict equilibria of these games exhibit \citep{Im05,MM09}. Motivated by these results, we are finally in a position to state our analogue of the folk theorem for stochastically perturbed congestion models:

\begin{theorem}
\label{thm:stability}
Strict Wardrop equilibria are stochastically asymptotically stable in the replicator dynamics (\ref{eq:SLRD}).
\end{theorem}

\begin{proof}
By relabeling indices if necessary, assume that $q=\sum_{i} \rho_{i} e_{i,0}$ is the strict Wardrop equilibrium of $\model$. Inspired by the deterministic setting and the original idea of \cite{Im05}, we will show that $H_{q}$ is a local stochastic Lyapunov function, i.e. that $\gen H_{q}(x) \leq - kH_{q}(x)$ for some $k>0$ and for all $x$ sufficiently close to $q$. Our result will then follow from Theorem 4 in \citet[pp. 314\textendash315]{GS71}.

To that end, consider a perturbed flow $x=\sum_{i,\alpha} x_{i\alpha} q_{i\alpha}$ close to $q$:
\begin{equation}
x_{i,0} = \rho_{i} (1-\eps_{i}),\quad
x_{i\mu} = \eps_{i} \xi_{i\mu} \text{ for } \mu=1,2\ldots\in\act_{i}\exclude\{0\}
\end{equation}
where $\eps_{i}>0$ controls the $L^{1}$ distance between $x_{i}$ and $q_{i}$, and $\xi_{i}$ is a point in the face of $\strat_{i}$ lying opposite to $q$ (i.e. $\xi_{i\mu}\geq0$ and $\sum_{\mu} \xi_{i\mu}=\rho_{i}$). Then, in view of Lemma \ref{lem:strictadjoint} in Appendix \ref{apx:adjoint}, the adjoint potential $L_{q}$ will be bounded below by:
\begin{equation}
L_{q}(x)\geq\insum_{i} \rho_{i}\eps_{i}\Delta\omega_{i},
\end{equation}
where $\Delta\omega_{i}=\min_{\mu}\{\omega_{i\mu}(q) - \omega_{i,0}(q)\}>0$ (recall that $q$ is a strict Wardrop equilibrium). Therefore, since the second term of $\gen H_{q}(x;\lambda)$ in (\ref{eq:stochentropy}) is clearly of order $\bigoh(\eps^{2})$, we obtain:
\begin{equation}
\gen H_{q}(x;\lambda) \leq
-\insum_{i}\rho_{i}\eps_{i}\Delta\omega_{i} + \bigoh(\eps^{2})
\end{equation}
where $\eps^{2}=\sum_{i}\eps_{i}^{2}$. On the other hand, we also have:
\begin{equation}
H_{q}(x;\lambda)
= \insum_{i}\frac{\rho_{i}}{\lambda_{i}}\log\frac{\rho_{i}}{x_{i,0}}
=-\insum_{i}\frac{\rho_{i}}{\lambda_{i}}\log(1-\eps_{i})
=\insum_{i}\frac{\rho_{i}}{\lambda_{i}}\eps_{i} + \bigoh(\eps^{2}).
\end{equation}
Thus, if we pick some positive $k<\min_{i}\{\lambda_{i} \Delta\omega_{i}\}$, some elementary algebra gives:
\begin{equation}
\gen H_{q}(x;\lambda) \leq -k\insum_{i}\frac{\rho_{i}}{\lambda_{i}} \eps_{i} + \bigoh(\eps^{2})
= -kH_{q}(x;\lambda) + \bigoh(\eps^{2}),
\end{equation}
thus showing that (\ref{eq:stochentropy}) holds whenever $\eps$ is small enough.
\end{proof}

In other words, Theorem \ref{thm:stability} implies that trajectories which start sufficiently close to a strict equilibrium will remain in its vicinity and will eventually converge to it with arbitrarily high probability. Nonetheless, this is a \emph{local} result: if the users' initial traffic distribution is not close to a strict equilibrium itself, Theorem \ref{thm:stability} does not apply; specifically, if $X(0)$ is an arbitrary initial condition in $\strat$, we cannot even tell if the trajectory $X(t)$ will \emph{ever} approach $q$.

To put this in more precise terms, it will be convenient to measure distances in $\strat$ with the $L^{1}$-norm: $\left\Vert\sum_{\alpha}z_{\alpha} e_{\alpha}\right\Vert_{1} = \sum_{\alpha}|z_{\alpha}|$. In this norm, it is not too hard to see that $\strat$ has a diameter of $2\sum_{i}\rho_{i}$, so pick some positive $\delta<2\sum_{i}\rho_{i}$ and let $K_{\delta}=\{x\in\strat: \|x-q\|_{1}\leq\delta\}$ be the corresponding compact neighbourhood of $q$. Then, to see if $X(t)$ ever hits $K_{\delta}$, we will examine the hitting time $\tau_{\delta}$:
\begin{equation}
\tau_{\delta}
\equiv \tau_{K_{\delta}}
=\inf\{t>0: X(t)\in K_{\delta}\}
=\inf\{t>0: \|X(t)-q\|_{1}\leq\delta\}.
\end{equation}
Thereby, our chief concern is this: is the hitting time $\tau_{\delta}$ finite with high probability? And if it is, is its expected value also finite?

To make our lives easier, let us consider the collective expressions:
\begin{align}
\label{eq:aggregates}
\rho &= \insum_{i}\rho_{i},
&\Delta\omega&=\rho^{-1} \insum_{i} \rho_{i} \Delta\omega_{i}\\
\sigma^{2} &=\insum_{r}\sigma_{r}^{2},
&\lambda&=\rho^{-1} \insum_{i}\rho_{i}\lambda_{i}.\notag
\end{align}
where $\Delta\omega_{i} = \min_{\mu\neq\alpha_{i}}\{\omega_{i\mu}(q) - \omega_{i\alpha_{i}}(q)\}>0$ is the minimum delay difference between a user's equilibrium path $\alpha_{i}$ and his other choices. We then have:
\begin{theorem}
\label{thm:timestrict}
Let $q=\sum_{i}\rho_{i}e_{i\alpha_{i}}$ be a strict Wardrop equilibrium of a congestion model $\model$ and assume that the users' learning rates satisfy the condition:
\begin{equation}
\label{eq:slowlearning}
\lambda\sigma^{2}<\Delta\omega.
\end{equation}
Then, for any $\delta<2\rho$ and any initial condition $X(0) = x\in\strat$ with finite relative entropy $H_{q}(x;\lambda)<\infty$, the hitting time $\tau_{\delta}$ has finite mean:
\begin{equation}
\label{eq:timestrict}
\ex_{x}[\tau_{\delta}]
\leq \frac{2H_{q}(x;\lambda)}{\Delta\omega} \frac{2\rho}{\delta(2\rho-\delta)}.
\end{equation}
\end{theorem}

\begin{proof}
As in the case of Theorem \ref{thm:stability}, our proof hinges on the expression:
\begin{equation}
\label{eq:LHq}
-\gen H_{q}(x;\lambda)
= L_{q}(x)
- \frac{1}{2}\insum_{i}\frac{\lambda_{i}}{\rho_{i}}\insum^{i}_{\beta,\gamma} \sigma_{\beta\gamma}^{2}(x_{i\beta}-q_{i\gamma})(x_{i\gamma}-q_{i\gamma}),
\end{equation}
where $q=\insum_{i}\rho_{i}e_{i,0}$ is the strict equilibrium in question. In particular, set $x=q+\theta z$, where $z=\insum_{i,\alpha} z_{i\alpha}e_{i\alpha}\in T_{q}\strat$ is the ``inward'' direction:
\begin{equation}
\txs
z_{i,0} = -\rho_{i},\quad
z_{i\mu}\geq0 \text{ for } \mu=1,2\ldots\in\act_{i}^{*}\equiv\act_{i}\exclude \{0\} \text{ and } \insum_{\mu}z_{i\mu}=\rho_{i}.
\end{equation}
Then, regarding the first term of (\ref{eq:LHq}), Lemma \ref{lem:strictadjoint} in Appendix \ref{apx:adjoint} readily yields:
\begin{equation}
L_{q}(q+\theta z) \geq \Phi(q+\theta z) - \Phi(q)
\geq \theta\insum_{i}\rho_{i}\Delta\omega_{i}
\text{  for all $\theta\in[0,1]$}.
\end{equation}
In a similar vein, the second term of (\ref{eq:LHq}) becomes:
\begin{flalign}
\insum^{i}_{\beta,\gamma}&\sigma_{\beta\gamma}^{2}(x_{i\beta} - q_{i\beta})(x_{i\gamma} - q_{i\gamma})
=\theta^{2} \insum^{i}_{\beta,\gamma} \sigma_{\beta\gamma}^{2} z_{i\beta} z_{i\gamma}\\
&=\theta^{2}\insum_{r}\insum^{i}_{\beta,\gamma} \sigma_{r}^{2} P^{i}_{r\beta} P^{i}_{r\gamma} z_{i\beta} z_{i\gamma}
= \theta^{2}\insum_{r} \sigma_{r}^{2} w_{ir}^{2},\notag
\end{flalign}
where $w_{i} = P^{i}(z)$. Since $w_{ir}\leq \rho_{i}$ for all $r\in\edges$, we then obtain the inequality:
\begin{equation}
\label{eq:LHestimate}
-\gen H_{q}(x;\lambda)
\geq \theta \insum_{i}\rho_{i} \Delta\omega_{i} - \frac{\theta^{2}}{2}\sigma^{2} \insum_{i}\lambda_{i}\rho_{i}
= (\rho\Delta\omega) \theta - \frac{\theta^{2}}{2}\rho\lambda \sigma^{2}
\end{equation}
where $\rho, \sigma^{2}$ and $\lambda,\Delta\omega$ are the respective aggregates and averages of (\ref{eq:aggregates}).

Suppose now that the rates $\lambda_{i}$ satisfy (\ref{eq:slowlearning}), i.e. $\lambda\sigma^{2}< \Delta\omega$. In that case, the RHS of (\ref{eq:LHestimate}) will be increasing for all $\theta\in[0,1]$ and we will have:
\begin{equation}
-\gen H_{q}(x;\lambda)
\geq\rho\left[
\theta\Delta\omega  - \frac{1}{2}\lambda\sigma^{2}\theta^{2}
\right]>0
\end{equation}
for all $x$ with $\|x-q\|_{1}\geq \theta \|z\|_{1} = 2\theta\sum_{i}\rho_{i}=2\rho \theta$. So, if $\|x-q\|_{1}\geq\delta$, we get:
\begin{equation}
-\gen H_{q}(x;\lambda)
\geq
\frac{\delta}{2} \Delta\omega - \frac{1}{2}\frac{\delta^{2}}{4\rho^{2}} \sigma^{2} \rho\lambda
\geq\frac{\delta}{2} \Delta\omega\left(1-\frac{\delta}{2\rho}\right)
>0.
\end{equation}
Therefore, if $K_{\delta}$ is the compact neighbourhood $K_{\delta} = \{x\in\strat: \|x-q\|_{1}\leq\delta\}$, we will have $\gen H_{q}(x)\leq-\frac{\delta}{2} \Delta\omega\left(1-\frac{\delta}{2\rho}\right)<0$ for all $x\notin K_{\delta}$. Then, by a simple (but very useful!) estimate of \citet[Theorem~5.3 in page~268]{Du96} we get:
\begin{equation}
\ex_{x}[\tau_{\delta}]
\leq \frac{2 H_{q}(x;\lambda)}{\Delta\omega (1-\delta/2\rho)} \frac{1}{\delta}.\qedhere
\end{equation}
\end{proof}

Recall now that Theorem \ref{thm:stability} ensures that a trajectory $X(t)$ which starts sufficiently close to a strict equilibrium $q$ will converge to $q$ with arbitrarily high probability. Therefore, since Theorem \ref{thm:timestrict} shows that $X(t)$ will come arbitrarily close to $q$ in finite time, a tandem application of these two theorems yields:

\begin{corollary}
\label{cor:stoconvergence}
If $q$ is a strict equilibrium of $\model$ and the players' learning rates $\lambda_{i}$ satisfy (\ref{eq:slowlearning}), the trajectories $X(t)$ converge to $q$ almost surely.
\end{corollary}

Of course, if a strict Wardrop equilibrium exists, then it is the unique equilibrium of the game (Proposition \ref{prop:uniquestrict}). In that case, Corollary \ref{cor:stoconvergence} is the stochastic counterpart of Theorem \ref{thm:detconvergence}: if the players' learning rates are soft enough compared to the level of the noise, then the solution orbits of the stochastic replicator dynamics (\ref{eq:SRD}) converge to a stationary traffic distribution almost surely. A few remarks are thus in order:

\begin{remark}[Temperance and Temperature]
Condition (\ref{eq:slowlearning}) shows that the rep\-li\-cator dynamics reward \emph{patience}: players who take their time in learning the game manage to weed out the noise and eventually converge to equilibrium. This begs to be compared with the (inverse) temperature analogy for the learning rates: if the ``learning temperature'' $T=1/\lambda$ is too low, the players' learning scheme becomes very rigid and this intemperance amplifies any random variations in the experienced delays. On the other hand, when the temperature rises above the threshold $T_{c} = \sigma^{2}/\Delta\omega$, the stochastic fluctuations are toned down and the deterministic drift draws users to equilibrium.
\end{remark}

\begin{remark}
Admittedly, the form of (\ref{eq:timestrict}) is a bit opaque for practical purposes. To lighten it up, note that we are only interested in small $\delta$, so the term $2\rho/(2\rho-\delta)$ may be ignored to leading order. Therefore, if we also assume for simplicity that all players learn at the same rate $\lambda_{i}=\lambda$, we get:
\begin{equation}
\ex_{x}[\tau_{\delta}]
\leq\frac{2 h}{\lambda\Delta\omega}\frac{\rho}{\delta}
\end{equation}
where $h=\rho^{-1}\sum_{i}\rho_{i}\log(\rho_{i}/x_{i})$ is the ``average'' Kullback-Leibler distance between $x$ and $q$.

This very rough estimate is pretty illuminating on its own. First and foremost, it shows that our bound for $\ex_{x}[\tau_{\delta}]$ is inversely proportional to the learning rate $\lambda$, much the same as in the deterministic setting where $\lambda$ essentially rescales time to $\lambda t$. Moreover, because $\rho$ and $\delta$ are both $\bigoh(N)$, one might be tempted to think that our time estimates are \emph{intensive}, i.e. independent of $N$. However, since delays increase (nonlinearly even) with the aggregate load $\rho$, the dependence on $N$ is actually hidden in $\Delta\omega$ \textendash~this also shows that the learning rates $\lambda_{i}$ do not have to be $\bigoh\left(1/|\edges|\right)$-small in order to satisfy (\ref{eq:slowlearning}).
\end{remark}

\setcounter{remark}{0}

In any event, since strict equilibria do not always exist, we should return to the generic case of \emph{interior} equilibria $q\in\Int(\strat)$. We have already seen that these equilibria are not very well-behaved in stochastic environments: they are not stationary in (\ref{eq:SLRD}) and (\ref{eq:posinterior}) shows that $\gen H_{q}$ is actually \emph{positive} in their vicinity. Despite all that, if the network $\net$ is irreducible and the users' learning rates $\lambda_{i}$ are slow enough, we will see that \emph{the replicator dynamics (\ref{eq:SLRD}) admit a finite invariant measure which concentrates mass around the (unique) equilibrium of $\model$}.

To state this result precisely, a little more groundwork is required. First off, akin to the case of strict equilibria, it will be convenient to measure distances from $q$ with a scaled variant of the $L^{1}$ norm. In particular, let $S_{q}=\{z\in T_{q}\strat: q+z\in\bd(\strat)\}$. Since $\strat$ is convex, any $x\in\strat$ can be uniquely expressed as $x=q+\theta z$ for some $z\in S_{q}$ and some $\theta\in[0,1]$, so we define the \emph{projective distance} $\Theta_{q}(x)$ of $x$ from $q$ to be:
\begin{equation}
\label{eq:projective}
\Theta_{q}(x) = \theta \Leftrightarrow x=q+\theta z \text{ for some $z\in S_{q}$ and $0\leq\theta\leq 1$}.
\end{equation}
$\Theta_{q}$ is not a bona fide distance function by itself, but it closely resembles the $L^{1}$ norm: the ``projective balls'' $B_{\theta}=\{x:\Theta_{q}(x)\leq\theta\}$ are rescaled copies of $\strat$ ($S_{q}$ is the ``unit sphere'' in this picture), and the graph $\mathrm{gr}(\Theta_{q})\equiv\{(x,\theta)\in\strat\times\R: \Theta_{q}(x) = \theta\}$ of $\Theta_{q}$ is simply a cone over the polytope $\strat$.









In a similar vein, we define the \emph{essence} of a point $q\in\strat$ to be:
\begin{equation}
\label{eq:essence}
\txs
\ess (q) = \rho^{-1}\min\big\{\|P(z)\|:z\in S_{q}\big\},
\end{equation}
where $\|\cdot\|$ denotes the ordinary Euclidean norm and the factor of $\rho$ was included for scaling purposes. Comparably to $\red(\net)$, $\ess(q)$ measures redundancy (or rather, the lack thereof): $\ess(q)=0$ only if some direction $z\in S_{q}$ is null for $P$, i.e. only if $\net$ is reducible.\footnote{With some trivial modifications, these definitions can be extended to points in $\bd(\strat)$ as well. In particular, if $q$ is strict we get $\Theta_{q_{i}}(x_{i}) = \|z_{i}\|_{1}/2\rho_{i}$ and, as a consequence of Proposition \ref{prop:uniquestrict}, we will have $\ess(q)>0$ irrespective of whether $\net$ is reducible or not.} 

\smallskip

We are finally in a position to state and prove:

\vspace{-5pt}

\begin{theorem}
\label{thm:recurrence}
Let $q\in\Int(\strat)$ be an interior equilibrium of an irreducible congestion model $\model$, and assume that the users' learning rates satisfy the condition:
\begin{equation}
\label{eq:slowlearning2}
\lambda<\frac{4}{5}\frac{m\rho\kappa^{2}}{\sigma^{2}},
\text{where $m=\inf\{\phi_{r}'(y_{r}):r\in\edges, y\in P(\strat)\}$ and $\kappa=\ess(q)$.}
\end{equation}
Then, for any interior initial condition $X(0)=x\in\Int(\strat)$, the trajectories $X(t)$ are recurrent (a.s.) and their time averages are concentrated in a neighbourhood of $q$. Specifically, if $\Theta_{q}(\cdot)$ denotes the projective distance (\ref{eq:projective}) from $q$, then:
\begin{equation}
\label{eq:time-average}
\ex_{x}\left[
\frac{1}{t} \int_{0}^{t} \Theta_{q}^{2}(X(s)) \dd s
\right]
\leq\theta_{\lambda}^{2} + \bigoh\left(1/t\right), \text{ where }
\theta_{\lambda}^{2} = \frac{1}{4}\left(\frac{m\rho\kappa^{2}}{\lambda\sigma^{2}}-1\right)^{-1}.
\end{equation}
Accordingly, the transition probabilities of $X(t)$ converge in total variation to an invariant probability measure $\pi$ on $\strat$ which concentrates mass around $q$. In particular, if $B_{\theta}=\{x\in\strat:\Theta_{q}(x)\leq\theta\}$ is a ``projective ball'' around $q$, we have:
\begin{equation}
\label{eq:measure}
\pi(B_{\theta}) \geq 1-\theta_{\lambda}^{2}/\theta^{2}.
\end{equation}
\end{theorem}

Following \cite{Bh78}, recurrence here means that for every $\xi\in\Int(\strat)$ and every neighbourhood $U_{\xi}$ of $\xi$, the diffusion $X(t)$ has the property:
\begin{equation}
\label{eq:recurrence}
\prob_{x}\{X(t_{k})\in U_{\xi}\}=1,
\end{equation}
for some sequence of (random) times $t_{k}$ that increases to infinity. Hence, using the recurrence criteria of \cite{Bh78}, we will prove our claim by showing that $X(t)$ hits a compact neighbourhood of $q$ in finite time (this is the hard part), and that the generator of a suitably transformed process is elliptic.

\begin{figure}
\subfigure[Learning in an irreducible network.]{%
\label{subfig:convirr}%
\includegraphics[width=0.48\columnwidth]{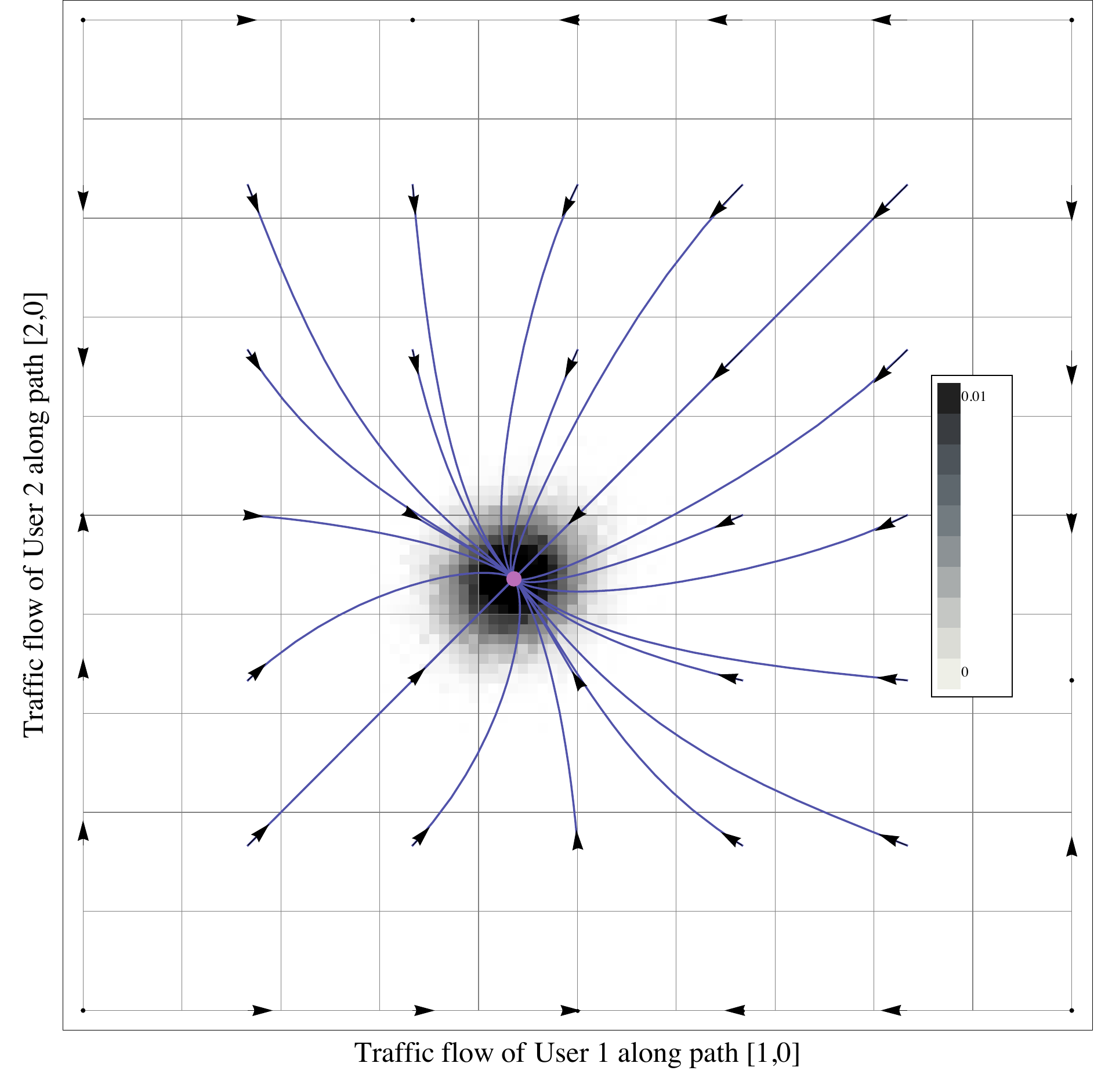}
}
\,
\subfigure[Learning in a reducible network.]{%
\label{subfig:convred}%
\includegraphics[width=0.48\columnwidth]{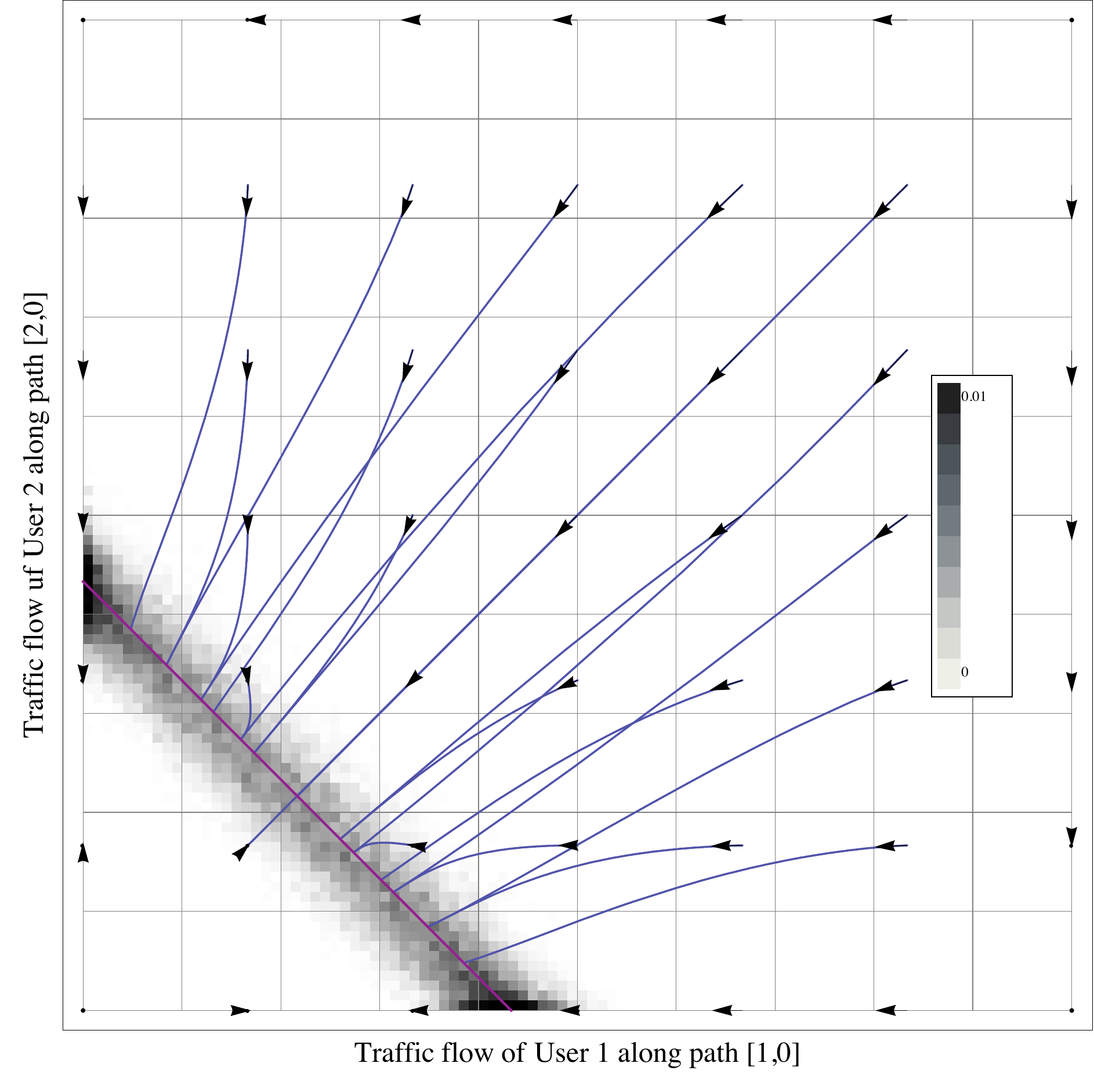}
}
\caption{Learning in the networks of Figure \ref{fig:network} with M/M/1 latency functions $\phi_{r}(y_{r})=(\mu_{r} - y_{r})^{-1}$ and arbitrarily chosen capacities $\mu_{r}$. The shades of gray represent the invariant distribution of (\ref{eq:SLRD}) (obtained by numerical simulations) and the flow lines (blue) are the solution trajectories of (\ref{eq:LRD}) -- in (b) they are actually projections because there is a third user as well.}
\label{fig:convergence}
\end{figure}

\begin{proof}[Proof of Theorem \ref{thm:recurrence}]

As we mentioned before, any $x\in\strat$ may be (uniquely) written in the ``projective'' form $x=q+\theta z$, where $\theta=\Theta_{q}(x)\in[0,1]$ is the projective distance of $x$ from $q$, and $z$ is a point in the ``sphere'' $S_{q}=\{z'\in T_{q}\strat:q+z'\in\bd(\strat)\}$. In this manner, (\ref{eq:stochentropy}) becomes:
\begin{flalign}
\label{eq:interiorLHq}
-\gen H_{q}(x;\lambda)=
L_{q} (q+\theta z)
&-\frac{1}{2}\insum_{i}\frac{\lambda_{i}}{\rho_{i}} \theta^{2}
\insum^{i}_{\beta,\gamma} \sigma_{\beta\gamma}^{2} z_{i\beta} z_{i\gamma}\\
&-\frac{1}{2}\insum_{i}\frac{\lambda_{i}}{\rho_{i}}
\insum^{i}_{\beta,\gamma} \sigma_{\beta\gamma}^{2} q_{i\beta} (\rho_{i}\delta_{\beta\gamma} - q_{i\gamma})\notag.
\end{flalign}

With regards to the first term of (\ref{eq:interiorLHq}), Lemma \ref{lem:interioradjoint} in Appendix \ref{apx:adjoint} and the definition (\ref{eq:essence}) of $\kappa$ yield $L_{q}(q+\theta z)\geq\frac{1}{2}m\|P(z)\|^{2}\theta^{2}\geq \frac{1}{2}m\kappa^{2}\rho^{2}\theta^{2}$. Moreover, we have already seen in the proof of Theorem \ref{thm:timestrict} that the second term of (\ref{eq:interiorLHq}) is bounded above:
\begin{equation}
\frac{1}{2}\sum_{i}\frac{\lambda_{i}}{\rho_{i}}\theta^{2}\insum^{i}_{\beta,\gamma} \sigma_{\beta\gamma}^{2}z_{i\beta} z_{i\gamma}\leq\frac{1}{2}\rho\lambda\sigma^{2}\theta^{2}.
\end{equation}

We are thus left to estimate the last term of (\ref{eq:interiorLHq}). To that end, (\ref{eq:posinterior}) gives:
\begin{flalign}
\insum^{i}_{\beta,\gamma}q_{i\beta} \left(\rho_{i}\delta_{\beta\gamma} - q_{i\gamma}\right)\sigma_{\beta\gamma}^{2}
= \insum_{r}\sigma_{r}^{2} y_{ir}(\rho_{i} - y_{ir})
\leq\frac{1}{4}\rho_{i}^{2}\sigma^{2}
\end{flalign}
where the last inequality stems from the bound $y_{ir}(\rho_{i}-y_{ir})\leq\frac{1}{4}\rho_{i}^{2}$ (recall that $0\leq y_{ir}\leq\rho_{i}$). Combining all of the above, we then get:
\begin{equation}
\label{eq:estimateLHq}
-\gen H_{q}(x;\lambda)
\geq \frac{1}{2}m\kappa^{2}\rho^{2}\theta^{2}
-\frac{1}{2}\rho\lambda\sigma^{2}\theta^{2}
-\frac{1}{8}\rho\lambda\sigma^{2}
\equiv g(\theta),
\quad 0\leq \theta\leq 1.
\end{equation}
As a result, if $\lambda<\frac{4}{5}\lambda_{0}$ where $\lambda_{0}=\frac{m\rho\kappa^{2}}{\sigma^{2}}$, it is easy to see that the RHS of (\ref{eq:estimateLHq}) will be increasing for $0\leq \theta\leq 1$. Moreover, it will also be positive for $\theta_{\lambda}<\theta\leq 1$, where $\theta_{\lambda}$ is the positive root of $g$: $\theta_{\lambda}=\frac{1}{2}(\lambda_{0}/\lambda-1)^{-1/2}$.

So, pick some positive $a<g(1)=\frac{1}{2}\rho\sigma^{2}\left(\lambda_{0}-\frac{5}{4}\lambda\right)$ and consider the set $K_{a}=\{q+\theta z:z\in S_{q}, g(\theta)\leq a\}$. By construction, $K_{a}$ is a compact neighbourhood of $q$ which does not intersect $\bd(\strat)$ and, by (\ref{eq:estimateLHq}), we have $\gen H_{q}(x;\lambda)\leq - a$ outside $K_{a}$. Therefore, if $\tau_{a}\equiv \tau_{K_{a}}$ denotes the hitting time $\tau_{a}=\inf\{t:X(t)\in K_{a}\}$, Theorem 5.3 in \citet[p.~268]{Du96} yields:
\begin{equation}
\label{eq:interiortime}
\ex_{x}[\tau_{a}] \leq \frac{H_{q}(x;\lambda)}{a} < \infty
\end{equation}
for every interior initial condition $X(0)=x\in\Int(\strat)$.

Inspired by a trick of \cite{Im05}, let us now consider the transformed process $Y(t)=\Psi(X(t))$ given by $\Psi_{i\mu}(x) = \log x_{i\mu}/x_{i,0}$, $\mu\in\act_{i}^{*}\equiv\act_{i}\exclude\{\alpha_{i,0}\}$. With $\frac{\pd \Psi_{i\mu}}{\pd x_{i\mu}} = 1/x_{i\mu}$ and $\frac{\pd \Psi_{i\mu}}{\pd x_{i,0}} = -1/x_{i,0}$, It\^o's formula gives:
\begin{equation}
\label{eq:dY}
dY_{i\mu}
=\gen \Psi_{i\mu}(X) \dd t
+\dd U_{i\mu} - \dd U_{i,0}
=\gen\Psi_{i\mu}(X)\dd t +  \insum_{r} Q^{i}_{r\mu} \sigma_{r} \dd W_{r},
\end{equation}
where $Q^{i}_{r\mu} = P^{i}_{r\mu} - P^{i}_{r,0}$ are the components of the redundancy matrix $Q$ of $\net$ in the basis $\tilde e_{i\mu} = e_{i\mu} - e_{i,0}$ of $T_{q}\strat$ -- recall also (\ref{eq:Qcomponents}) and the relevant discussion in Section \ref{subsec:flows}. 

We now claim that the generator of $Y$ is elliptic. Indeed, if we drop the user index $i$ for convenience and set $A_{\mu r} = Q_{r\mu} \sigma_{r}$, $\mu\in\coprod_{i}\act_{i}^{*}$, it suffices to show that the matrix $AA^{T}$ is positive-definite. Sure enough, for any tangent vector $z=\sum_{\mu} z_{\mu} \tilde e_{\mu}\in T_{q}\strat$, we get:
\begin{equation}
\label{eq:eigenvalue}
\langle Az, Az\rangle
= \insum_{\mu,\nu} \left(AA^{T}\right)_{\mu\nu} z_{\mu} z_{\nu}
=\insum_{\mu,\nu}\insum_{r} Q_{r\mu} Q_{r\nu} \sigma_{r}^{2} z_{\mu} z_{\nu}
=\insum_{r} \sigma_{r}^{2}w_{r}^{2},
\end{equation}
where $w=Q(z)$. Since $\net$ is irreducible, we will have $w\neq 0$, and in view of (\ref{eq:eigenvalue}) above, this proves our assertion.

We have thus shown that the process $Y(t)$ hits a compact neighbourhood of $\Psi(q)$ in finite time (on average), and also that the generator of $Y$ is elliptic. From the criteria of \citet[Lemma~3.4]{Bh78} it follows that $Y$ is recurrent, and since $\Psi$ is invertible in $\Int(\strat)$, the same must hold for $X(t)$ as well. In a similar fashion, these criteria also ensure that the transition probabilities of the diffusion $X(t)$ converge in total variation to an invariant probability measure $\pi$ on $\strat$, thus proving the first part of our theorem.

To obtain the estimate (\ref{eq:time-average}), note that Dynkin's formula \cite[see e.g.][Theorem~7.4.1]{Ok07} applied to (\ref{eq:estimateLHq}) yields:
\begin{flalign}
\ex_{x}\left[H_{q}(X(t);\lambda)\right]
&= H_{q}(x;\lambda)
+\ex_{x}\left[
\int_{0}^{t} \gen H_{q}(X(s);\lambda) \dd s
\right]\\
&\leq H_{q}(x;\lambda)
-\frac{1}{2} \rho\sigma^{2} (\lambda_{0} - \lambda) \ex_{x}\left[
\int_{0}^{t} \Theta_{q}^{2}(X(s)) \dd s
\right]
+\frac{1}{8} \rho\lambda\sigma^{2} t,\notag
\end{flalign}
and with $\ex_{x}[H_{q}(X(t);\lambda)]\geq 0$, we easily get:
\begin{equation}
\label{eq:time-average2}
\ex_{x}\left[
\frac{1}{t}\int_{0}^{t}
\Theta_{q}^{2}(X(s))\dd s
\right]
\leq \theta_{\lambda}^{2} + \frac{C}{t},
\text{ where }
C = \frac{2}{\rho\sigma^{2}}\frac{H_{q}(x;\lambda)}{\lambda_{0}-\lambda}.
\end{equation}

We are thus left to establish the bound $\pi(B_{\theta}) \geq 1 - \theta_{\lambda}^{2}/\theta^{2}$ which shows that the invariant measure $\pi$ concentrates its mass around the ``projective balls'' $B_{\theta}$. For that, we will use the ergodic property of $X(t)$, namely that:
\begin{equation}
\pi(B_{\theta})
= \lim_{t\to\infty}\ex_{x}\left[\frac{1}{t}\int_{0}^{t}\chi_{B_{\theta}}(X(s))\dd s\right],
\end{equation}
where $\chi_{B_{\theta}}$ is the indicator function of $B_{\theta}$. However, with $\Theta_{q}^{2}(x)/\theta^{2}\geq 1$ outside $B_{\theta}$ by definition, it easily follows that:
\begin{equation}
\ex_{x}\left[
\frac{1}{t}\int_{0}^{t} \chi_{B_{\theta}}(X(s))\dd s
\right]
\geq \ex_{x}\left[
\frac{1}{t}\int_{0}^{t}\left(1-\Theta_{q}^{2}(X(s))\Big/\theta^{2}\right) \dd s
\right]
\end{equation}
and the bound (\ref{eq:measure}) follows by letting $t\to\infty$ in (\ref{eq:time-average2}).
\end{proof}

We conclude this section with a few remarks on our results so far:

\begin{remark}[{\em Learning vs. Noise}]
The nature of our bounds reveals a most interesting feature of the replicator equation (\ref{eq:SLRD}). On the one hand, as $\lambda\to 0$, we also get $\theta_{\lambda}\to 0$ and the invariant measure $\pi$ converges vaguely to a point mass at $q$. Hence, if the learning rate $\lambda$ is slow enough (or if the noise $\sigma$ is low enough), we recover Theorem \ref{thm:detconvergence} (as we should!). On the other hand, there is a clear downside to using very slow learning rates: the expected time to hit a neighbourhood of an equilibrium is inversely proportional to $\lambda$. As a result, choosing learning rates is a delicate process and users will have to balance the rate versus the desired sharpness of their convergence.
\end{remark}

\begin{remark}[{\em The Effect of Redundancy}]
The irreducibility assumption is actually quite important: it appears both in the ``slow-learning'' condition (\ref{eq:slowlearning2}) (recall that $\ess(q)=0$ if $q$ is an interior point of a reducible network) and also in the proof that the generator of $Y(t) = \Psi(X(t))$ is elliptic. This shows that the stochastic dynamics (\ref{eq:SLRD}) are not obvlivious to redundant degrees of freedom, in stark contrast with the deterministic case (Theorem \ref{thm:detconvergence}).

Regardless, we expect that an analogue for Theorem \ref{thm:recurrence} still holds for reducible networks if we replace $q$ with the entire (affine) set $\eq$. More precisely, we conjecture that under a suitably modified learning condition, the transition probabilities of $X(t)$ converge to an invariant distribution which concentrates mass around $\eq$ (see Fig.~\ref{subfig:convred}). One way to prove this claim would be to find a suitable way to ``quotient out'' $\ker Q$ but, since the replicator equation (\ref{eq:SLRD}) is not invariant over the redundant fibres $x+\ker Q, x\in\strat$, we have not yet been able to do so.
\end{remark}

\begin{remark}[{\em Sharpness}]
We should also note here that the bounds we obtained are not the sharpest possible ones. For example, the learning condition (\ref{eq:slowlearning2}) can be tightened and the assumption that $\phi_{r}'>0$ can actually be dropped. In that case however, the corresponding expressions would become significantly more complicated without adding adding much essence, so we have opted to keep our analysis focused on the simpler estimates.
\end{remark}



\section{Discussion}
\label{sec:discussion}

In this last section, we will discuss some issues that have not been thoroughly addressed in the rest of the paper and provide some directions for future work.

\paragraph{Learning and Optimality}

We have already noted that the traffic flows which minimise the aggregate latency $\omega(x) = \sum_{i}\rho_{i}\omega_{i}(x)$ in a network correspond precisely to the Wardrop equilibria of a congestion model which is defined over the same network and whose delay functions are given by the ``marginal latencies'' $\phi_{r}^{*}(y_{r}) = \phi_{r}(y_{r}) + y_{r}\phi_{r}'(y_{r})$ \citep[see e.g.][]{RT02}. Hence, if we set $\omega_{i\alpha}^{*}(x) = \sum^{i}_{\alpha}P^{i}_{r\alpha}\phi_{r}^{*}(y_{r})$ and substitute $\omega_{i\alpha}^{*}$ instead of $\omega_{i\alpha}$ in the replicator dynamics (\ref{eq:LRD}) and (\ref{eq:SLRD}), our analysis yields:

\begin{theorem}
Let $\model\equiv\model(\net,\phi)$ be a congestion model with strictly convex latency functions $\phi_{r}$, $r\in\edges$, and assume that users follow a replicator learning scheme with cost functions $\omega_{i\alpha}^{*}$. Then:
\begin{enumerate}
\item In the deterministic case (\ref{eq:LRD}), players converge to a traffic flow which minimises the aggregate delay $\omega(x) = \sum \rho_{i}\omega_{i}(x)$.
\item In the stochastic case (\ref{eq:SLRD}), if the network is irreducible and the players' learning rates are slow enough, their time-averaged flows will be concentrated near the (necessarily unique) optimal distribution $q$ which minimises $\omega$.
\end{enumerate}
\end{theorem}

Of course, for a sharper statement one need only reformulate Theorems \ref{thm:detconvergence}, \ref{thm:stability}, and \ref{thm:recurrence} accordingly (the convexity of $\phi_{r}^{*}$ replaces the monotonicity requirement for $\phi_{r}$). The only thing worthy of note here is that the marginal costs $\phi_{r}^{*}(y_{r})$ do not really constitute ``local information'' that users can acquire simply by routing their traffic and recording the delays that they experience. However, the missing components $y_{r}\phi_{r}'(y_{r})$ can easily be measured by observers monitoring the edges of the network and could be subsequently publicised to all users that employ the edge $r\in\edges$. Consequently, if the adminstrators of a network wish users to figure out the optimal traffic allocation on their own, they simply have to go the (small) extra distance of providing such monitors on the network's links.

\paragraph{Equilibrium Classes}
In a certain sense, interior and strict equilibria represent the extreme ends of the Wardrop spectrum, so it was a reasonable choice to focus our analysis on them. Nevertheless, there are equilibrium classes that we did not consider: for instance, there are pure Wardrop equilibria which are not strict, or there could be ``quasi-strict'' equilibria $q$ in the boundary of $\strat$ with the property that $\omega_{i\alpha}(q)>\omega_{i}(q)$ for all $\alpha$ which are not present in $q$.

Strictly speaking, such equilibria are not covered by either Theorem \ref{thm:stability} or Theorem \ref{thm:recurrence}. Still, by a suitable modification of our stochastic calculations, we may obtain similar convergence and stability results for these types of equilibria as well. For example, modulo a ``slow-learning'' condition similar to (\ref{eq:slowlearning2}), it is easy to see that pure equilibria that are not strict are still stochastically stable. The reason we have opted not to consider all these special cases is that it would be too much trouble for little gain: the assortment of similar-looking results that we would obtain in this way would confuse things more than it would clarify them.

\paragraph{Exponential Learning}
In the context of $N$-person Nash games, we have already mentioned that the replicator dynamics also arise as the result of an ``exponential learning'' process, itself a variant of logistic fictitious play \citep{FL98,MM09}. The way this scheme works is that players keep cumulative scores of their strategies' performance and they employ each strategy with a probability which is exponentially proportional to these scores. As such, it is not too hard to adapt this method directly to our congestion setting.

In more detail, assume that all users $i\in\play$ keep performance scores $V_{i\alpha}$ of the paths at their disposal as specified by the differential equation:
\begin{equation}
\label{eq:detscore}
dV_{i\alpha}(t) = -\omega_{i\alpha}(x(t)) \dd t,
\end{equation}
where $x(t)$ is the traffic profile at time $t$. Based on these scores, the users then update their traffic flows according to the Boltzmann distribution:
\begin{equation}
\label{eq:Boltzmann}
x_{i\alpha}(t) = \frac{e^{\lambda_{i}V_{i\alpha}(t)}}{\sum^{i}_{\beta}e^{\lambda_{i}V_{i\beta}(t)}},
\end{equation}
where $\lambda_{i}$ denotes the learning rate of player $i\in\play$ (this expression also explains why these rates can be seen as inverse temperatures). In this way, by decoupling these expressions, one obtains the deterministic replicator equation (\ref{eq:LRD}).

We thus see that exponential learning tells us nothing new in deterministic environments. In the presence of noise however, the scores $V_{i\alpha}$ also reflect any fluctuations in the observed delays, so we obtain instead:
\begin{equation}
\label{eq:stoscore}
dV_{i\alpha}(t) = -\omega_{i\alpha}(X) \dd t + \sigma_{i\alpha}\dd U_{i\alpha},
\end{equation}
where, as in (\ref{eq:U}), $\dd U_{i\alpha}$ describes the total noise along the path $\alpha\in\act_{i}$. Therefore, if the users' flow profile $X(t)$ is updated according to (\ref{eq:Boltzmann}), It\^o's lemma now gives:
\begin{flalign}
\label{eq:CXRD}
d X_{i\alpha}
&= \lambda_{i} X_{i\alpha}\left[\omega_{i}(X) - \omega_{i\alpha}(X)\right]\dd t
+ \lambda_{i} X_{i\alpha}\left[d U_{i\alpha} - \rho_{i}^{-1}\insum^{i}_{\beta} X_{i\beta} \dd U_{i\beta}\right]\notag\\
+&\frac{\lambda_{i}^{2}}{2} X_{i\alpha}\left[
\insum^{i}_{\beta}(\rho_{i}\delta_{\alpha\beta} - 2 X_{i\beta})\sigma_{\alpha\beta}^{2}
- \rho_{i}^{-1}\insum^{i}_{\beta,\gamma} \sigma_{\beta\gamma}^{2} X_{i\gamma}(\rho_{i}\delta_{\beta\gamma} - 2 X_{i\beta})
\right]dt.
\end{flalign}

As far as the rationality properties of these new dynamics are concerned, a simple modification in the proof of Theorem \ref{thm:stability} suffices to show that strict Wardrop equilibria are stochastically stable in (\ref{eq:CXRD}). Just the same, the extra drift term in (\ref{eq:CXRD}) complicates things considerably, so results containing explicit estimates of hitting times are significantly harder to obtain. Of course, this approach might well lead to improved convergence rates, but since the calculations would take us too far afield, we prefer to postpone this analysis for the future.

\paragraph{The Brown-von~Neumann-Nash Dynamics}
Another powerful learning scheme is given by the Brown-von~Neumann-Nash (BNN) dynamics \citep[see e.g.][]{FL98} where users look at the ``excess delays''
\begin{equation}
\label{eq:excess}
\psi_{i\alpha}(x)
= \left[\omega_{i}(x) - \omega_{i\alpha}(x)\right]^{+}
=\max\left\{\omega_{i}(x) - \omega_{i\alpha}(x),0\right\}
\end{equation}
and update their traffic flows according to the differential equation:
\begin{equation}
\label{eq:BNN}
\frac{d x_{i\alpha}}{dt} = \psi_{i\alpha}(x(t)) - \psi_{i}(x(t)),
\end{equation}
where $\psi_{i}(x)=\rho_{i}^{-1}\sum^{i}_{\alpha} x_{i\alpha}\psi_{i\alpha}(x)$. On the negative side, these dynamics require users to monitor delays even along paths that they do not employ. On the other hand, they satisfy the pleasant property of ``non-complacency'' \citep{Sa01}: the stationary states of (\ref{eq:BNN}) coincide with the game's Wardrop equilibria and every solution trajectory converges to a connected set of such equilibria.

In terms of convergence to a Wardrop equilibrium, Theorem \ref{thm:detconvergence} shows that the replicator dynamics behave at least as well as the BNN dynamics (except perhaps on the boundary of $\strat$), so there is no real reason to pick the more complicated expressions (\ref{eq:excess}), (\ref{eq:BNN}). However, this might not be true in the presence of stochastic fluctuations: in fact, virtually nothing is known about the behaviour of the BNN dynamics in stochastic environments so this question alone makes pursuing this direction a worthwhile project.


\section*{Acknowledgements}

We would like to extend our gratitude to M.~Scarsini for a series of fruitful discussions on Braess's paradox which helped us clarify the differences between the equilibrial conditions that arise in congestion models.


\appendix

\section{Properties of the Adjoint Potential}
\label{apx:adjoint}

We collect here some of the most useful properties of the \emph{adjoint potential}:
\begin{equation}
\tag{\ref{eq:adjoint}}
L_{q}(x)
\equiv \insum_{i,\alpha}(x_{i\alpha} - q_{i\alpha}) \omega_{i\alpha}(x)
= \insum_{r} (y_{r}-y^{*}_{r}) \phi_{r}(y_{r})
\equiv \Lambda (y).
\end{equation}
To begin with, the equality in (\ref{eq:adjoint}) stems from the invariance identity:
\begin{equation}
\label{eq:invariance}
\insum_{\alpha} z_{\alpha}\omega_{\alpha}
= \insum_{\alpha}\insum_{r} z_{\alpha} P_{r\alpha} \phi_{r}
= \insum_{r} w_{r} \phi_{r},
\quad z\in V
\end{equation}
where $P:V\to W$ is the indicator matrix of the network $\net$ and $w=P(z)$. It is then easy to verify that $L_{q}(x) = L_{q}(x')$ whenever $x'-x\in\ker Q$ (thus justifying (\ref{eq:adjoint})), and also that $L_{q}=L_{q'}$ iff $q'-q\in\ker Q$. As a result, the notation $L_{q}(x)\equiv\Lambda(y)$ is consistent with any choice of $q\in\eq$.

This ``adjoint'' potential owes its name to the formula for integration by parts:
\begin{equation}
\insum_{r}\int_{y_{r}^{*}}^{y_{r}} \phi_{r}(w) \dd w
= \insum_{r}(y_{r}-y_{r}^{*})\phi_{r}(y_{r})
-\insum_{r}\int_{y_{r}^{*}}^{y_{r}} w \phi_{r}'(w) \dd w.
\end{equation}
Since the latencies $\phi_{r}$ are increasing, this expression immediately yields the estimate (\ref{eq:adjoint}): $\Phi(y)- \Phi(y^{*}) \leq \Lambda(y)$. However, if we also assume that $q$ is strict (say $q=\sum_{i} \rho_{i} e_{i,0}$ for convenience), we can get a more direct bound:
\begin{lemma}
\label{lem:strictadjoint}
Let $q=\sum_{i}\rho_{i} e_{i,0}$ be a strict Wardrop equilibrium and let $z\in T_{q}\strat$. Then, for all $t\geq0$ such that $x(t)= q+tz\in\strat$, we have:
\begin{equation}
\label{eq:strictadjoint}
L_{q}(q+tz)\geq \frac{1}{2} \insum_{i}  \Delta\omega_{i} \|z_{i}\|_{1} t,\quad
\text{where $\Delta\omega_{i} = \txs\min_{\mu\neq0} \{\omega_{i\mu}(q)-\omega_{i,0}(q)\}$ .}
\end{equation}
\end{lemma}

\begin{proof}
Clearly, to have $q+tz\in\strat$ for some $t>0$, $z$ must be of the form:
\begin{equation}
\label{eq:possiblez}
z=\insum_{i}z_{i}
\text{ with }
z_{i} = \insum^{i}_{\mu} z_{i\mu}(e_{i\mu} - e_{i,0})
\text{ and }
z_{i\mu}\geq0
\text{ for all }\mu\in\act_{i}^{*}\equiv\act_{i}\exclude\{0\}.
\end{equation}
So, let $f(t) = \Phi(y(t))$ where $y(t) = P(x(t)) = y^{*}+tw$ and $w=P(z)$. With $\Phi$ convex, we get $f(t)\geq f(0) + f'(0)t$ and a simple differentiation yields: $f'(0) =\left.\frac{d}{dt}\right|_{t=0}\insum_{r} \Phi_{r}(y_{r}^{*}+tw_{r}) = \insum_{r} w_{r} \phi_{r}(y_{r}^{*})$. However, thanks to (\ref{eq:invariance}) and (\ref{eq:possiblez}), we may rewrite this sum as:
\begin{flalign}
\insum_{r} w_{r} \phi_{r}(y_{r}^{*})
&= \insum_{i}\insum^{i}_{\alpha} z_{i\alpha} \omega_{i\alpha}(q)
=\insum_{i}\insum^{i}_{\mu} z_{i\mu} \left[\omega_{i\mu}(q) - \omega_{i,0}(q)\right]\\
&\geq\insum_{i}\insum^{i}_{\mu} z_{i\mu} \Delta\omega_{i}
=\frac{1}{2} \insum_{i} \Delta\omega_{i} \|z_{i}\|_{1},\notag
\end{flalign}
because $\|z_{i}\|_{1} = \sum^{i}_{\alpha}|z_{i\alpha}| = \left|-\sum^{i}_{\mu} z_{i\mu}\right| + \sum^{i}_{\mu} |z_{i\mu}| = 2 \sum^{i}_{\mu} z_{i\mu}$. We thus obtain:
\begin{equation}
L_{q}(q+tz)
=\Lambda(y^{*}+tw)
\geq f(t)-f(0)
\geq \frac{1}{2} \insum_{i}  \Delta\omega_{i} \|z_{i}\|_{1} t.
\qedhere
\end{equation}
\end{proof}

This lemma shows that $L_{q}$ increases at least linearly along all ``inward'' rays $q+tz$.\footnote{It is interesting to note here the relation with Proposition \ref{prop:uniquestrict}: if the ray $q+tz$ is inward-pointing, then $z$ cannot be ``redundant'', i.e. we cannot have $z\in\ker P$.} This is not so if $q$ is an \emph{interior} equilibrium: 

\begin{lemma}
\label{lem:interioradjoint}
Let $q\in\Int(\strat)$ be an interior Wardrop equilibrium and let $z\in T_{q}\strat$. Then, for all $t\geq0$ such that $x(t)=q+tz\in\strat$, we have:
\begin{equation}
\label{eq:interioradjoint}
\txs
L_{q}(q+tz)\geq \frac{1}{2} m\|P(z)\|^{2} t^{2},\quad
\text{where $m=\inf\{\phi_{r}'(y_{r}):r\in\edges, y\in P(\strat)\}$.}
\end{equation}
\end{lemma}

\begin{proof}
Following the proof of Lemma \ref{lem:strictadjoint} above, we obtain:
\begin{equation}
f'(0)
=\insum_{i}\insum^{i}_{\alpha} z_{i\alpha}\omega_{i\alpha}(q)
=\insum_{i}\insum^{i}_{\alpha} z_{i\alpha} \omega_{i}(q)
=0
\end{equation}
where the second equality follows from the fact that $q$ is an interior equilibrium (that is, $\omega_{i\alpha}(q) = \omega_{i}(q)$ for all paths $\alpha\in\act_{i}$), and the last one is a consequence of $z$ being tangent to $\strat$ (meaning that $\sum^{i}_{\alpha}z_{i\alpha}=0$). On the other hand, we also get:
\begin{equation}
f''(t)
=\frac{d^{2}}{dt^{2}}\insum_{r}\Phi_{r}(y_{r}^{*}+tw_{r})
= \insum_{r} w_{r}^{2} \phi_{r}'(y_{r}^{*} + tw_{r}).
\end{equation}
Clearly, since the set $P(\strat)$ of load profiles $y$ is compact and the (continuous) functions $\phi_{r}'$ are positive, we will also have $m=\inf\{\phi_{r}'(y_{r}):y\in P(\strat),r\in\edges\}>0$. We will thus have $f(t)\geq\frac{1}{2}mt^{2}$, and a first order Taylor expansion with Lagrange remainder easily yields:
\begin{equation}
\txs
L_{q}(q+tz)
=\Lambda(y^{*}+tw)
\geq f(t)-f(0)
\geq \frac{1}{2} m\|P(z)\|^{2}t^{2}.
\qedhere
\end{equation}
\end{proof}

\section{Stochastic Calculations}
\label{apx:stochastic}

This appendix is devoted to the calculations that are hidden under the hood of (\ref{eq:stochentropy}), the equation that describes the evolution of the relative entropy $H_{q}(x;\lambda)$.

\begin{proof}[Proof of Lemma \ref{lem:stochentropy}]
Let $V_{q}(t)= H_{q}(X(t);\lambda)$. We then have:
\begin{flalign}
\label{eq:dV1}
dV_{q}
&= \insum_{i,\alpha} \frac{\pd H_{q}}{\pd x_{i\alpha}} \dd X_{i\alpha}
+\frac{1}{2}\insum_{i,\alpha}\insum_{j,\beta} \frac{\pd^{2} H_{q}}{\pd x_{i\alpha} \pd x_{j\beta}} \left(dX_{i\alpha}\right)\ip\left(dX_{j\beta}\right)\\
&=-\insum_{i,\alpha}\frac{1}{\lambda_{i}} \frac{q_{i\alpha}}{x_{i\alpha}} \dd X_{i\alpha}
+\frac{1}{2} \insum_{i,\alpha} \frac{1}{\lambda_{i}} \frac{q_{i\alpha}}{x_{i\alpha}^{2}}
\left(d X_{i\alpha}\right)^{2}\notag.
\end{flalign}
However, with $X(t)$ being as in (\ref{eq:SLRD}), we readily obtain:
\begin{multline}
\label{eq:dX2}
\left(d X_{i\alpha}\right)^{2}
= \lambda_{i}^{2} X_{i\alpha}^{2} \left(
d U_{i\alpha} - \rho_{i}^{-1} \insum^{i}_{\beta} X_{i\beta} \dd U_{i\beta}
\right)^{2}\\
= \lambda_{i}^{2} X_{i\alpha}^{2} \left[
\left(dU_{i\alpha}\right)^{2}
-\frac{2}{\rho_{i}} \insum^{i}_{\beta} X_{i\beta} \dd U_{i\alpha}\ip \dd U_{i\beta}
+\frac{1}{\rho_{i}^{2}}\insum^{i}_{\beta,\gamma} X_{i\beta} X_{i\gamma} \dd U_{i\beta}\ip\dd U_{i\gamma}
\right]\\
=\lambda_{i}^{2} X_{i\alpha}^{2} \left[
\sigma_{i\alpha}^{2}
-\frac{2}{\rho_{i}} \insum^{i}_{\beta} \sigma_{\alpha\beta}^{2} X_{i\beta}
+\frac{1}{\rho_{i}^{2}} \insum^{i}_{\beta,\gamma} \sigma_{\beta\gamma}^{2} X_{i\beta} X_{i\gamma}
\right] dt.
\end{multline}
As a result, we may combine the two equations (\ref{eq:dV1}) and (\ref{eq:dX2}) to obtain:
\begin{flalign}
\label{eq:dV2}
dV_{q}=-
&
\insum_{i,\alpha} q_{i\alpha} \left[\omega_{i}(X) - \omega_{i\alpha}(X)\right]dt
-\insum_{i,\alpha} q_{i\alpha}\left[
d U_{i\alpha} - \rho_{i}^{-1} \insum^{i}_{\beta} X_{i\beta}\dd U_{i\beta}
\right]\\
+\frac{1}{2}
&
\insum_{i,\alpha}\lambda_{i} q_{i\alpha}\!\left[
\sigma_{i\alpha}^{2}
-\frac{2}{\rho_{i}} \insum^{i}_{\beta} \sigma_{\alpha\beta}^{2} X_{i\beta}
+\frac{1}{\rho_{i}^{2}} \insum^{i}_{\beta,\gamma} \sigma_{\beta\gamma}^{2} X_{i\beta} X_{i\gamma}
\right]dt.\notag
\end{flalign}
Therefore, if we focus at a particular user $i\in\play$, the last term of (\ref{eq:dV2}) gives:
\begin{flalign}
\label{eq:dV3}
&\insum^{i}_{\alpha}
q_{i\alpha}\left[
\sigma_{i\alpha}^{2}
-\frac{2}{\rho_{i}} \insum^{i}_{\beta} \sigma_{\alpha\beta}^{2} X_{i\beta}
+\frac{1}{\rho_{i}^{2}}\insum^{i}_{\beta,\gamma} \sigma_{\beta\gamma}^{2} X_{i\beta} X_{i\gamma}
\right]\notag\\
&=\insum^{i}_{\alpha} q_{i\alpha} \sigma_{i\alpha}^{2}
- \frac{2}{\rho_{i}} \insum^{i}_{\alpha,\beta} \sigma_{\alpha\beta}^{2} q_{i\alpha} X_{i\beta}
+\frac{1}{\rho_{i}} \insum^{i}_{\beta,\gamma}\sigma_{\beta\gamma}X_{i\beta} X_{i\gamma}\notag\\
&=\insum^{i}_{\alpha}q_{i\alpha} \sigma_{i\alpha}^{2} - \frac{1}{\rho_{i}} \insum^{i}_{\beta,\gamma} q_{i\beta} q_{i\gamma} \sigma_{\beta\gamma}^{2}\notag\\
&+\frac{1}{\rho_{i}}\left[
\insum^{i}_{\beta,\gamma} q_{i\beta} q_{i\gamma} \sigma_{\beta\gamma}^{2}
-2\insum^{i}_{\beta,\gamma} X_{i\beta} q_{i\gamma} \sigma_{\beta\gamma}^{2}
+\insum^{i}_{\beta,\gamma} X_{i\beta} X_{i\gamma} \sigma_{\beta\gamma}^{2}
\right]\notag\\
&=\frac{1}{\rho_{i}}\left[
\insum^{i}_{\beta,\gamma} q_{i\beta} (\rho_{i}\delta_{\beta\gamma} - q_{i\gamma}) \sigma_{\beta\gamma}^{2}
+ \insum^{i}_{\beta,\gamma} \sigma_{\beta\gamma}^{2} (X_{i\beta} - q_{i\beta})(X_{i\gamma} - q_{i\gamma})
\right]
\end{flalign}
and the lemma follows by substituting (\ref{eq:dV3}) into (\ref{eq:dV2}) and keeping only the resulting drift, that is, the first and third terms of (\ref{eq:dV2}).
\end{proof}


\bibliographystyle{imsart-nameyear}
\bibliography{Bibliography}

\end{document}